\documentclass[11pt]{article}

\usepackage{amsmath}    
\usepackage{graphicx}   
\usepackage{verbatim}   
\usepackage{color}      
\usepackage{hyperref}   

\usepackage{amssymb}
\usepackage{amsthm}

\usepackage[left=2.5cm, right=2.5cm, top=2cm, bottom=2cm]{geometry}
\theoremstyle{theorem}
\newtheorem{Def}{Definition}[section]

\newtheorem{Prop}[Def]{Proposition}
\newtheorem{Lem}[Def]{Lemma}
\newtheorem{Thm}[Def]{Theorem}

\newtheorem{Ass}[Def]{Assumption}
\newtheorem{Hyp}[Def]{Hypothesis}
\theoremstyle{definition}
\newtheorem{Rem}[Def]{Remark}

\newcommand{\p}{\mathbb{P}}
\newcommand{\e}{\mathbb{E}}
\newcommand{\real}{\mathbb{R}}
\newcommand{\n}{\mathbb{N}}

\newcommand{\1}{{\bf 1}}

\newcommand{\rd}{\mathrm{d}}

\newcommand{\x}{\textbf{x}}
\newcommand{\y}{\textbf{y}}
\newcommand{\J}{\textbf{J}}
\newcommand{\g}{\textbf{g}}

\newcommand{\ochi}{\overline{\chi}}

\allowdisplaybreaks

\begin{document}

\title{
Semi-implicit Euler-Maruyama approximation for non-colliding particle systems 
}
\author{
	Hoang-Long Ngo\footnote{Hanoi National University of Education, 136 Xuan 
		Thuy - Cau Giay - Hanoi - Vietnam, email: ngolong@hnue.edu.vn}
	$\quad $ and $\quad$ 
	Dai Taguchi\footnote{Osaka University, 1-3, Machikaneyama-cho, Toyonaka, Osaka 560-8531, Japan, email: dai.taguchi.dai@gmail.com}
}
\maketitle
\begin{abstract}
We introduce a semi-implicit Euler-Maruyama approximation which preservers the non-colliding property for some class of non-colliding particle systems such as Dyson Brownian motions, Dyson-Ornstein-Uhlenbeck processes and Brownian particles systems with nearest neighbour repulsion, and study its rates of convergence in both $L^p$-norm and  path-wise sense. \\
\textbf{2010 Mathematics Subject Classification}: 60H35; 41A25; 60C30\\
\textbf{Keywords}:
Implicit Euler-Maruyama approximation $\cdot$
non-colliding particle systems$\cdot$
rate of convergence
\end{abstract}


\section{Introduction}

Let $X=(X(t)=(X_1(t),\ldots,X_d(t))^{*})_{t \geq 0}$ be a solution of the following system of stochastic differential equations (SDEs)
\begin{align} \label{def:X}
	\mathrm{d}X_i(t)
	=\left\{
	\sum_{j \not = i}\frac{\gamma_{i,j}}{X_i(t) - X_j(t)}
	+ b_i(X(t) )
	\right\}dt
	+ \sum_{j=1}^d \sigma_{i,j}(X(t))\mathrm{d}W_j(t),
	~i=1,\ldots,d,
\end{align}
with $X(0) \in \Delta_d=\{\x=(x_1,\ldots,x_d)^{*} \in \real^d: x_1 < x_2 <\cdots < x_d\}$,  $\gamma_{i,j} = \gamma_{j,i} \geq 0$, and $W=(W(t)=(W_1(t),\ldots,W_d(t))^{*})_{t \geq 0}$  a $d$-dimensional standard Brownian motion defined on a probability space $(\Omega,\mathcal{F},\p)$ with a filtration $(\mathcal{F}_t)_{ t \geq 0}$ satisfying the usual conditions.

The systems of SDEs \eqref{def:X} are used to model the stochastic evolution of $d$ particles with electrostatic repulsion and restoring force. An interesting feature of these systems is their deep connection with the theory of eigenvalue distribution of  randomly-diffusing symmetric matrices and Jack symmetric polynomials (see \cite{Dyson, Bru89, Bru90, Katori, Ra}). 
The existence and uniqueness of  a strong non-colliding solution to such kind of systems have been studied intensively by many authors (see \cite{RogerShi, CepaLepingle, GrMa14, Inukai} and the references therein). However, there are still few results on the numerical approximation for these systems, in spite of their practical importance. To the best of our knowledge, the paper of Li and Menon \cite{LiMe} is the only work in this direction. These authors introduced an explicit tamed Euler-Maruyama approximation for Dyson Brownian motion and studied its consistency via a couple of numerical experiments. However, their scheme, unfortunately, does not preserve the non-colliding property of solution, which is an important characteristic of the Dyson Brownian motion.

Many authors have studied the numerical approximation for one-dimensional SDEs with boundary (Bessel process and Cox-Ingersoll-Ross (CIR) process).
Dereich, Neuenkirch, and Szpruch \cite{DeNeSz} introduced an implicit Euler-Maruyama scheme for CIR process and showed that the rate of convergence is $1/2$.
That result was extended to one-dimensional SDEs with boundary condition by Alfonsi \cite{Al} and Neuenkirch and Szpruch \cite{NeSz}. It was proved that if the drift coefficient is one-sided Lipschitz and  smooth, then the implicit Euler-Maruyama scheme is well defined and converges to the unique solution in $L^p$ sense with convergence rate of order  $1/2$ or $1$ provided that the boundaries are not accessible. In the case of CIR and Bessel processes with accessible boundaries, the rates of strong convergence of discrete approximation schemes may be very slow (see Hutzenthaler et. al \cite{HJN} and Hefter and Jentzen \cite{HJ}).  It should be noted that if we consider $d=2, \ b_i = 0$ and $(\sigma_{i,j})_{1\leq i,j \leq d}$ is a diagonal and constant matrix, then $X_2 - X_1$ is a Bessel process. The numerical approximation for multidimensional SDEs with boundary has been studied by Gy\"ongy \cite{Gyongy} and Jentzen et. al  \cite{JKN}. These authors introduced various explicit and implicit Euler-Maruyama schemes and studied their convergence in the path-wise sense. 

The main aims of this paper are to introduce a numerical approximation method which preserves the non-colliding property of solution to the system \eqref{def:X} and to study its strong rate of convergence both in $L^p$-norm and in path-wise sense. To the best of our knowledge, this is the first paper to discuss the strong rate of approximation for multidimensional stochastic differential equations whose solution stays in a domain. Note that the singular coefficients $\frac{1}{X_i - X_j}$ make the system difficult to deal with.  In order to overcome this obstacle, we need an upper bound for both moments and inverse moments of $X_i - X_j$. 

The remainder of this paper is organised as follows. In the next section, we introduce a semi-implicit Euler-Maruyama approximation $X^{(n)}$ for equation \eqref{def:X} and study its consistency. More precisely, we first show  the rate that $X^{(n)}$ converges to $X$ is of order $1/2$ in the path-wise sense. Then under some key conditions on the integrability of $X$, we show that the rate is of order almost $1/2$ in the $L^p$-norm. Finally, under further conditions on the regularity of $b_i$, we show that the rate is of order $1$ in the $L^p$-norm. In Section \ref{Sec:moment}, we study some generalized classes of interacting Brownian particle systems and Brownian particles with nearest neighbour repulsion. We first show the existence and uniqueness for the solution of these systems and then we show that the solution satisfies the key integrability condition which allows us to obtain the rates of convergence of $X^{(n)}$. In the Appendix, we discuss how to compute the implicit scheme in some particular cases.

\section{Approximation for non-colliding processes}\label{sec_semi_EM} \label{Sec:2}

Throughout this paper, we suppose that the following assumptions hold. 
\begin{Ass}\label{Ass_1}
	\begin{enumerate}
		\item[(A1)] $X(0) \in \Delta_d$ almost surely. 
		\item[(A2)] The parameters $\gamma_{i,j}$ are non-negative constants satisfying $\gamma_{i,j} = \gamma_{j,i}$ for $i,j =1, \ldots,d$ with $i \neq j$ and $\gamma_{i,i+1}>0$ for $i=1,\ldots,d-1$.
		\item[(A3)] The  coefficients $b_i:\real^d \to \real$, $i=1,\ldots, d$ are globally Lipschitz continuous, that is,
		\begin{align*}
			\|b\|_{Lip}
			:= \sup_{i=1,\ldots,d} \sup_{x\not = y} \frac{|b_i(x) - b_i(y)|}{|x-y|}
			< \infty.
		\end{align*}
		\item [(A4)] The coefficients $\sigma_{i,j}:\real^d \to \real$, $i,j=1,\ldots,d$ are globally Lipschitz continuous and bounded, that is,
		\begin{align*}
			\|\sigma\|_{Lip}&:= \sup_{i,j=1,\ldots,d} \sup_{x\not = y} \frac{|\sigma_{i,j}(x) - \sigma_{i,j}(y)|}{|x-y|}
			< \infty,\\
			\sigma_d^2
			&:=\sup_{i=1,\ldots,d} \sup_{x \in \real^d} \sum_{k=1}^d \sigma_{i,k}(x)^2<\infty.
		\end{align*}
	\end{enumerate}
\end{Ass}

\subsection{Explicit Euler-Maruyama scheme}

Let us first consider the explicit Euler-Maruyama approximation for non-colliding particle system \eqref{def:X} which is defined by $\widetilde{X}^{(n)}(0)=X(0)$ and for $t \in (0,T]$ and $i = 1, \ldots, d,$ 
\begin{align*}
	\mathrm{d}\widetilde{X}_i^{(n)}(t)
	&=\left\{\sum_{j \not = i}\frac{\gamma_{i,j}}{\widetilde{X}_i^{(n)}(\eta_n(t))-\widetilde{X}_j^{(n)}(\eta_n(t))} +b_i\left(\widetilde{X}^{(n)}(\eta_n(t))\right) \right\}\mathrm{d}t\\
	&\quad+ \sum_{j=1}^{d}\sigma_{i,j}\left(\widetilde{X}^{(n)}(\eta_n(t))\right) \mathrm{d} W_j(t),
\end{align*}
where $\eta _n(s) = kT/n=:t_k^{(n)}$ if $ s \in \left[kT/n, (k+1)T/n \right)$.
For $X(0) \in \Delta_d$, the explicit Euler-Maruyama scheme is well-defined. 
Since, for each $i=1,\ldots, {d-1}$,  the quantity
\begin{align*}
	&\widetilde{X}^{(n)}_{i+1}(t_1^{(n)}) - \widetilde{X}_i^{(n)}(t_1^{(n)})\\
	&= \widetilde{X}_{i+1}(0)-\widetilde{X}_{i}(0)\\
	&\quad+\left\{
	\sum_{k \neq i+1}
	\frac{\gamma_{i+1,k}}{\widetilde{X}_{i+1}(0)-\widetilde{X}_{k}(0)}
	-\sum_{k \neq i}
	\frac{\gamma_{i,k}}{\widetilde{X}_{i}(0)-\widetilde{X}_{k}(0)}
	+b_{i+1}(\widetilde{X}(0))-b_{i}(\widetilde{X}(0))
	\right\} \frac{T}{n}\\
	&\qquad +\sum_{j=1}^{d} \left\{ \sigma_{i+1,j}(\widetilde{X}(0))-\sigma_{i,j}(\widetilde{X}(0)) \right\}W_{j}(t_1^{(n)})
\end{align*}
is normally distributed provided that $\sigma_{i+1,j}(\widetilde{X}(0)) \not = \sigma_{i,j}(\widetilde{X}(0))$. This implies   
\begin{align*}
	\p\Big(\widetilde{X}^{(n)}(t^{(n)}_1) \in \Delta_d\Big)<1.
\end{align*}
Therefore the explicit Euler-Maruyama scheme is not suitable for approximating the non-colliding process \eqref{def:X}.

\subsection{Semi-implicit Euler-Maruyama scheme}
In the following we propose a semi-implicit Euler-Maruyama scheme for \eqref{def:X}, which preserves the non-colliding property of the solution. The construction of the semi-implicit scheme is based on the following result. 

\begin{Prop}\label{IEM_sol}
	Let $a=(a_1,\ldots,a_d)^{*} \in \real^d, \ c_{i,j} = c_{j,i} \geq 0$,  for any $1\leq i < j \leq d$ and $c_{i,i+1}>0$. The following system of equations has a unique solution,
	\begin{equation} \label{eqn_xi}
		\xi_i = a_i + \sum_{j\not = i} \frac{c_{i,j}}{\xi_i - \xi_j}, \quad i=1,\ldots, d,
	\end{equation}
	which satisfies $\xi_1<\xi_2<\cdots < \xi_d$.
\end{Prop}
\begin{proof} 
	The following proof is based on a homotopy argument presented in \cite[page 230]{EHairer}.
	Denote $\J = (1,2,\ldots, d)^* \in \Delta_d$ and 
	$$g_i(\x) = g_i(x_1,\ldots, x_d) = a_i - i + \sum_{j\not = i} \frac{c_{i,j}}{x_i - x_j}, \quad i=1,\ldots, d.$$
	Note that $\g=(g_1,\ldots, g_d)^* \in C^\infty(\Delta_d;\real^d)$. 
	We consider the following differential equation
	\begin{equation} \label{dyJ}
		\begin{cases} 
			\frac{d \x}{dt} = \frac{\partial \g(\x)}{\partial \x} \frac{d \x}{dt} + \g(\J), \quad t>0,\\
			\x(0)= \J,
		\end{cases}
	\end{equation}
	where $ \frac{\partial \g(\x)}{\partial \x} = (\frac{\partial g_i(\x)}{\partial x_j})_{i,j}.$
	Since $c_{i,j}=c_{j,i}$,  for any $\y=(y_1,\ldots,y_d)^* \in \real^d$, we have 
	\begin{align*}
		\langle \frac{\partial \g(\x)}{\partial \x} \y, \y \rangle &= \sum_{i,j} \frac{\partial g_i(\x)}{\partial x_j}y_jy_i \\
		&= \sum_{i \not = j} \frac{c_{i,j}}{(x_i-x_j)^2}y_jy_i -\sum_{i \not = j} \frac{c_{i,j}}{(x_i-x_j)^2}y_i^2\\
		&= - \frac 12 \sum_{i \not = j} \frac{c_{i,j}}{(x_i-x_j)^2}(y_i-y_j)^2 \leq 0. 
	\end{align*}
	Therefore, $I_d-\frac{\partial \g(\x)}{\partial \x}$ is a strictly positive definite matrix.
	Since $\g \in C^\infty(\Delta_d;\real^d)$, equation \eqref{dyJ} has a unique local solution which can be continued up to the boundary of $\Delta_d$.
	Denote $t^*=\inf\{t>0: \x(t) \not \in \Delta_d\}$. For $t<t^*$, thanks to the initial condition $\x(0) =\J$, we have 
	\begin{equation*} 
		\x(t) = \g(\x(t))+ \J + (t-1)\g(\J).
	\end{equation*} 
	Moreover, 
	$$\Big|\frac{d \x}{dt}\Big|^2 = \langle \frac{\partial \g(\x)}{\partial \x} \frac{d \x}{dt}, \frac{d \x}{dt}\rangle + \langle \g(\J), \frac{d \x}{dt} \rangle \leq \langle \g(\J), \frac{d \x}{dt} \rangle \leq \Big|\frac{d \x}{dt}\Big| |\g(\J)|.$$
	Thus 
	$$ \Big|\frac{d \x}{dt}\Big| \leq  |\g(\J)|.$$
	This estimation together with the fact that  $g(\x) $ blows up at the boundary of $\Delta_d$ implies that $t^*=\infty$. Let $t=1$, we get 
	$$\x(1) = \g(\x(1)) + \J \in \Delta_d,$$
	which means that $\xi = \x(1)$ is a solution to equation \eqref{eqn_xi}.
	
	Now we consider the uniqueness of solution to equation \eqref{eqn_xi} in $\Delta_d$.
	Let $\xi=(\xi_1,\ldots,\xi_d)$, $\mu=(\mu_1,\ldots,\mu_d) \in \Delta_d$ be solutions of the equation \eqref{eqn_xi}.
	Then, since $c_{i,j}=c_{j,i} \geq 0$, it follows from the identity
	\begin{align} \label{id1}
		\sum_{i=1}^{d} A_i \sum_{j \neq i} B_{i,j}=\sum_{i<j} \{A_iB_{i,j}+A_jB_{j,i}\},~A_i, B_{i,j} \in \real,
	\end{align}
	that
	\begin{align*}
		|\xi-\mu|^2
		&=\langle \xi-\mu, \xi-\mu \rangle\\
		&=\sum_{i=1}^{d} (\xi_i-\mu_i) \sum_{j \neq i} c_{i,j} \left\{ \frac{1}{\xi_i-\xi_j} - \frac{1}{\mu_i-\mu_j} \right\}\\
		&=\sum_{i<j} c_{i,j} \{ (\xi_i-\mu_i)-(\xi_j-\mu_j) \} \left\{ \frac{1}{\xi_i-\xi_j} - \frac{1}{\mu_i-\mu_j} \right\} \\
		&=\sum_{i<j} c_{i,j} \{ (\xi_i-\xi_j)-(\mu_i-\mu_j) \} \left\{ \frac{1}{\xi_i-\xi_j} - \frac{1}{\mu_i-\mu_j} \right\} 
		\leq 0.
	\end{align*}
	This concludes $\xi=\mu$.
\end{proof}
\begin{Rem}
	An interesting consequence of Proposition \ref{IEM_sol} is that the non-linear system of equations \eqref{eqn_xi} has exactly $d !$ solutions on $\real^d$.
\end{Rem}

\begin{Rem}
	The system of equations \eqref{eqn_xi} does not have a closed form solution in general. In Section \ref{Sec:4} we will construct an approximation scheme for its solution in some particular cases.
\end{Rem}

Based on Proposition \ref{IEM_sol}, a semi-implicit Euler-Maruyama scheme for non-colliding process \eqref{def:X} is defined as follows: $X^{(n)}(0):=X(0)$ and for each $k=0,\ldots,n-1$, $X^{(n)}(t_{k+1}^{(n)})$ is the unique solution in $\Delta_d$ of the following equation:
\begin{align*}
	X_i^{(n)}(t_{k+1}^{(n)})
	=X_i^{(n)}(t_{k}^{(n)})
	&+ \left\{ \sum_{j \neq i} \frac{\gamma_{i,j}}{X_i^{(n)}(t_{k+1}^{(n)})-X_j^{(n)}(t_{k+1}^{(n)})}+ b_i\left(X^{(n)}(t_{k}^{(n)})\right) \right\} \frac{T}{n}\\
	&+ \sum_{j=1}^{d} \sigma_{i,j} \left(X^{(n)}(t_{k}^{(n)})\right) \left\{ W_j(t_{k+1}^{(n)})-W_j(t_{k}^{(n)}) \right\}.
\end{align*}
We then define for $t \in (0,T] \setminus \{t_1^{(n)},\ldots,t_n^{(n)}\}$,
\begin{align*}
	X_i^{(n)}(t)
	=X_i^{(n)}(\eta_n(t))
	&+ \left\{ \sum_{j \neq i} \frac{\gamma_{i,j}}{X_i^{(n)}(\kappa_n(t))-X_j^{(n)}(\kappa_n(t))}+ b_i(X^{(n)}\left(\eta_n(t))\right) \right\} (t-\eta_n(t))\\
	&+ \sum_{j=1}^{d}\sigma_{i,j} \left(X^{(n)}(\eta_n(t))\right) \left\{ W_j(t)-W_j(\eta_n(t)) \right\},
\end{align*}
where $\kappa_n(s) = (k+1)T/n=t_{k+1}^{(n)}$ if $ s \in \left[kT/n, (k+1)T/n \right)$.
Hence $X^{(n)}(t)$ satisfies 
\begin{align*}
	X_i^{(n)}(t)
	=X_i(0)
	&+\int_{0}^{t}
	\left\{ \sum_{j \neq i} \frac{\gamma_{i,j}}{X_{i}^{(n)}(\kappa_n(s))-X_{j}^{(n)}(\kappa_n(s))}+ b_i(X^{(n)}\left(\eta_n(s))\right) \right\} \mathrm{d}s \notag\\
	&+\sum_{j=1}^{d} \int_{0}^{t} \sigma_{i,j}(X^{(n)}(\eta_n(s))) \mathrm{d}W_j(s).
\end{align*}
We denote $X_{i,j}(t)=X_i(t)-X_j(t)$ and $X_{i,j}^{(n)}(t):=X_i^{(n)}(t)-X_j^{(n)}(t)$.

We also repeatedly use the following representation of the estimation error, $e_i(t):=X_i(t)-X_i^{(n)}(t)$
and $e(t):=(e_1(t),\ldots,e_d(t))^{*}$.
Then for $k=0,\ldots,n-1$, we have
\begin{align}\label{def_ei}
	e_i(t_{k+1}^{(n)})
	&=e_i(t_{k}^{(n)})
	+\sum_{j \neq i}  \left\{\frac{\gamma_{i,j}}{X_{i,j}(t_{k+1}^{(n)})}-\frac{\gamma_{i,j}}{X_{i,j}^{(n)}(t_{k+1}^{(n)})} \right\} \frac{T}{n} \\
	&\quad+\left\{b_i(X(t_{k}^{(n)}))-b_i(X^{(n)}(t_{k}^{(n)}))\right\} \frac{T}{n}\notag\\
	&\quad+\sum_{j=1}^{d}
	\left\{
	\sigma_i(X(t_{k}^{(n)}))
	-\sigma_i(X^{(n)}(t_{k}^{(n)}))
	\right\}
	\{W_j(t_{k+1}^{(n)})-W_j(t_{k}^{(n)})\}
	+r_i(k) \notag,
\end{align}
where
\begin{align}\label{def_rik}
	r_i(k)
	&:= \sum_{j \neq i} \int_{t_{k}^{(n)}}^{t_{k+1}^{(n)}} \left\{\frac{\gamma_{i,j}}{X_{i,j}(s)}-\frac{\gamma_{i,j}}{X_{i,j}(t_{k+1}^{(n)})} \right\} \mathrm{d}s
	+\int_{t_{k}^{(n)}}^{t_{k+1}^{(n)}} \left\{b_i(X(s))-b_i(X(t_{k}^{(n)}))\right\} \mathrm{d}s \\
	&\quad+
	\sum_{j=1}^{d} \int_{t_{k}^{(n)}}^{t_{k+1}^{(n)}} \left\{\sigma_{i,j}(X(s)) -\sigma_{i,j}(X(t_{k}^{(n)})) \right\} \mathrm{d}W_j(s) \notag.
\end{align}

\subsection{The case of constant diffusion coefficient}

In this  subsection, we consider the convergence of $X^{(n)}$ where  diffusion coefficient is a constant. 

The following result states that $X^{(n)}$ converges to $X$ at the rate of order almost $1/2$ in the path-wise sense provided that the system \eqref{def:X} has a strong solution in $\Delta_d$ on $[0,T]$.

\begin{Thm} \label{pathwise1}
	Assume that $\sigma_{i,j}(x) \equiv \sigma_{i,j}$ and system of equations \eqref{def:X} has a unique strong solution in $\Delta_d$ on $[0,T]$. Then there exists a finite random variable $\eta$ which does not depend on $n$ such that 
	$$\sup_{k=1,\ldots, n} |X(t^{(n)}_k) - X^{(n)}(t^{(n)}_k)| \leq \frac{\sqrt{\log n}}{\sqrt{n}}\eta \quad \text{a.s.}$$
\end{Thm}
\begin{proof}
	Using the identity \eqref{id1}
	and the fact that  $e_j-e_i=X_{j,i}-X_{j,i}^{(n)}$, we get
	\begin{align*}
		&|e(t_{k+1}^{(n)})|^2
		=\sum_{i=1}^{d} |e_{i}(t_{k+1}^{(n)})|^2\\
		=&\sum_{i=1}^{d} e_i(t_{k}^{(n)}) e_i(t_{k+1}^{(n)})
		+ \sum_{i=1}^{d} e_i(t_{k+1}^{(n)}) \sum_{j \neq i} \left\{\frac{\gamma_{i,j}}{X_{i,j}(t_{k+1}^{(n)})}-\frac{\gamma_{i,j}}{X_{i,j}^{(n)}(t_{k+1}^{(n)})} \right\} \frac{T}{n}\\
		&+\sum_{i=1}^{d} e_i(t_{k+1}^{(n)}) \left\{b_i(X(t_{k}^{(n)}))-b_i(X^{(n)}(t_{k}^{(n)}))\right\} \frac{T}{n}
		+\sum_{i=1}^{d} e_i(t_{k+1}^{(n)}) r_i(k)\\
		=&\sum_{i=1}^{d} e_i(t_{k}^{(n)}) e_i(t_{k+1}^{(n)})
		+ \sum_{i<j} \{X_{j,i}(t_{k+1}^{(n)})-X_{j,i}^{(n)}(t_{k+1}^{(n)})\} \left\{\frac{\gamma_{i,j}}{X_{j,i}(t_{k+1}^{(n)})}-\frac{\gamma_{i,j}}{X_{j,i}^{(n)}(t_{k+1}^{(n)})} \right\} \frac{T}{n}\\
		&+\sum_{i=1}^{d} e_i(t_{k+1}^{(n)}) \left\{b_i(X(t_{k}^{(n)}))-b_i(X^{(n)}(t_{k}^{(n)}))\right\} \frac{T}{n}
		+\sum_{i=1}^{d} e_i(t_{k+1}^{(n)}) r_i(k).
	\end{align*}
	Using the fact that $(x-y)(\frac{1}{x}-\frac{1}{y}) \leq 0$ and  $xy \leq x^2/2+y^2/2$, we have
	\begin{align*}
		|e(t_{k+1}^{(n)})|^2
		\leq &\frac{1}{2}|e(t_{k}^{(n)})|^2
		+\frac{1}{2}|e(t_{k+1}^{(n)})|^2\\
		&\qquad	+ |e(t_{k}^{(n)})| \sum_{i=1}^{d} |e_i(t_{k+1}^{(n)})| \frac{T\|b\|_{Lip}}{n}
		+\sum_{i=1}^{d} e_i(t_{k+1}^{(n)}) r_i(k).
	\end{align*}
	Hence we have, for any $k=0,\ldots,n-1$,
	\begin{align*}
		|e(t_{k+1}^{(n)})|^2
		\leq &|e(t_{k}^{(n)})|^2
		+|e(t_{k}^{(n)})| \sum_{i=1}^{d} |e_i(t_{k+1}^{(n)})| \frac{C_1}{n}
		+2\sum_{i=1}^{d} |e_i(t_{k+1}^{(n)})| |r_i(k)|,\\
		\leq& \sum_{\ell=0}^{k} |e(t_{\ell}^{(n)})| \sum_{i=1}^{d} |e_i(t_{\ell+1}^{(n)})| \frac{	C_1}{n}
		+2 \sum_{\ell=0}^{k} \sum_{i=1}^{d} |e_i(t_{\ell+1}^{(n)})| |r_i(\ell)|,
	\end{align*}
	where $C_1:=2T\|b\|_{Lip}$.
	By taking the supremum with respect to $k$, we obtain for any $m=1,\ldots,n$
	\begin{align*}
		&\sup_{k=1,\ldots, m}|e(t_{k}^{(n)})|^2
		\leq \sum_{\ell=0}^{m-1} |e(t_{\ell}^{(n)})| \sum_{i=1}^{d} |e_i(t_{\ell+1}^{(n)})| \frac{C_1}{n}
		+2 \sum_{\ell=0}^{m-1} \sum_{i=1}^{d} |e_i(t_{\ell+1}^{(n)})| |r_i(\ell)|\\
		\leq& \sup_{k=1,\ldots, m}|e(t_{k}^{(n)})| \sum_{\ell=0}^{m-1} \sup_{k=1,\ldots, \ell}|e(t_{k}^{(n)})| \frac{dC_1}{n}
		+2 \sup_{k=1,\ldots, m}|e(t_{k}^{(n)})| \sum_{\ell=0}^{m-1} \sum_{i=1}^{d} |r_i(\ell)|.
	\end{align*}
	and thus,
	\begin{align*}
		\sup_{k=1,\ldots, m}|e(t_{k}^{(n)})|
		\leq& \sum_{\ell=0}^{m-1} \sup_{k=1,\ldots, \ell}|e(t_{k}^{(n)})| \frac{dC_1}{n}
		+2 \sum_{\ell=0}^{m-1} \sum_{i=1}^{d} |r_i(\ell)|.
	\end{align*}
	By using discrete Gronwall's inequality (e.g. Chapter XIV, Theorem 1 and Remark 1,2 in \cite{MiPeFi}, page 436-437), we obtain,
	\begin{align}
		\sup_{k=1,\ldots, m}|e(t_{k}^{(n)})|
		\leq& 
		2\left\{
		1+
		\sum_{\ell=0}^{m-1} \frac{dC_1}{n} \exp\left( \sum_{\ell=0}^{m-1} \frac{dC_1}{n} \right)
		\right\}
		\sum_{\ell=0}^{m-1} \sum_{i=1}^{d} |r_i(\ell)| \notag
		\\
		\leq &
		2\left\{
		1+dC_1 \exp \left( dC_1 \right)
		\right\}
		\sum_{\ell=0}^{m-1} \sum_{i=1}^{d} |r_i(\ell)|. \label{est_sup_e}
	\end{align}
	Therefore,
	\begin{align*}
		&\dfrac{ \sup_{k=1,\ldots, n}|e(t_{k}^{(n)})|}{2\left\{
			1+dC_1 \exp \left( dC_1 \right)
			\right\}}\\
		&\leq  \sum_{i=1}^{d} \sum_{j \neq i} \int_{0}^{T} \left|\frac{\gamma_{i,j}}{X_{i,j}(s)}-\frac{\gamma_{i,j}}{X_{i,j}(\kappa_n(s))} \right| \mathrm{d}s
		+\sum_{i=1}^{d}
		\int_{0}^{T} \left|b_i(X(s))-b_i(X(\eta_n(s)))\right| \mathrm{d}s\\ 
		&\leq  \sum_{i=1}^{d} \sum_{j \neq i} \int_{0}^{T} \gamma_{i,j} \left|\frac{X_{i,j}(s)- X_{i,j}(\kappa_n(s))}{\inf_{0\leq s\leq T} X_{i,j}(s)^2} \right| \mathrm{d}s
		+d \|b\|_{Lip}
		\int_{0}^{T} \left| X(s) - X(\eta_n(s)) \right| \mathrm{d}s.
	\end{align*}
	Moreover, for any $0 \leq s \leq t\leq T$,
	\begin{align*}
		|X_i(t) - X_i(s)|
		&\leq \int_s^t  \left\{\sum_{j \not = i} \frac{ \gamma_{i,j} }{|X_i(u) - X_j(u)|}+\|b\|_{Lip} |X(u)|+ |b_i(0)|\right\}du\\
		&\quad+ \sum_{j=1}^d |\sigma_{i,j}||W_j(t) - W_j(s)|.
	\end{align*}
	Since $X_t\in \Delta_d$ for $t\in[0,T]$, we have
	\begin{align*}
		\sup_{u\in [0,T]} \left\{\sum_{j \not = i} \frac{ \gamma_{i,j} }{|X_i(u) - X_j(u)|} +\|b\|_{Lip} |X(u)|+ |b_i(0)|\right\} < \infty,
	\end{align*}
	and
	$$ \inf_{0\leq s \leq T} \inf_{i \not =  j} X_{i,j}(s)^2 > 0.$$
	These estimates together with L\'evy's modulus of continuity theorem yield the desired result.
\end{proof}

\begin{Rem}
	The class of SDEs \eqref{def:X} with $\sigma_{i,j}(x) \equiv \sigma_{i,j}$ contains both Dyson Brownian motions (e.g. \cite{Al,Dyson}) and Dyson-Ornstein-Uhlenbeck processes (e.g. \cite{Ra}).
\end{Rem}

In order to show the convergence of the semi-implicit Euler Maruyama scheme in $L^p$-norm, we need the following hypothesis on the integrablity and Kolmogorov type condition of $X$.

\begin{Hyp} \label{Ip}
	There exist constants $\hat{p}>0$ and $0<\hat{C}<\infty$ such that 
	$$\sup_{t\in[0,T]} \e[|X(t)|^{\hat{p}}] +  \max_{0\leq i < d}\sup_{t\in[0,T]} \e[|X_{i,i+1}(t)|^{-\hat{p}}] < \hat{C},$$
	and 
	$$ \e[|X(t)-X(s)|^{\hat{p}}] \leq \hat{C}|t-s|^{\hat{p}/2}, \quad \text{for all  } \ 0 \leq s < t \leq T.$$ 
\end{Hyp}

In Section \ref{Sec:moment} we will introduce some conditions on $\gamma_{i,j}$, $b_i$ and $\sigma_{i,j}$, which guarantee that Hypothesis \ref{Ip} holds.

\begin{Thm}\label{main_1}
	Suppose that the assumptions of Theorem \ref{pathwise1} hold. Moreover, suppose that Hypothesis \ref{Ip} holds for some $\hat{p} = 3p\geq 3$. 
	Then there exists $C>0$ which depends on $d$ such that for any $n \in \n$,
	\begin{align*}
		\e\left[\sup_{k=1,\ldots,n}|X(t_{k}^{(n)})-X^{(n)}(t_{k}^{(n)})|^p\right]^{1/p}
		\leq \frac{C}{n^{1/2}}.
	\end{align*}
\end{Thm}
\begin{proof}
	We will use the estimate \eqref{est_sup_e} to show the desired result. Note that from \eqref{def_rik} we get
	\begin{align} \label{est_ril}
		\sum_{\ell=0}^{m-1} \sum_{i=1}^{d} |r_i(\ell)|
		&\leq  \sum_{i=1}^{d} \sum_{j \neq i} \int_{0}^{t_{m}^{(n)}} \left|\frac{\gamma_{i,j}}{X_{i,j}(s)}-\frac{\gamma_{i,j}}{X_{i,j}(\kappa_n(s))} \right| \mathrm{d}s\\
		&\quad 	+\sum_{i=1}^{d}
		\int_{0}^{t_{m}^{(n)}} \left|b_i(X_i(s))-b_i(X_i(\eta_n(s)))\right| \mathrm{d}s. \notag 
	\end{align}
	
	It follows from H\"older's inequality that
	\begin{align*}
		&\e\left[\left(\int_{0}^{t_{m}^{(n)}} \left | \frac{1}{X_{i,j}(s)}-\frac{1}{X_{i,j}(\kappa_n(s))} \right| \mathrm{d}s\right)^{p}\right]
		\leq T^{p-1} \int_{0}^{T} \e\left[ \frac{|X_{i,j}(s)-X_{i,j}(\kappa_n(s))|^p}{|X_{i,j}(s)|^p |X_{i,j}(\kappa_n(s))|^p} \right] \mathrm{d} s\\
		&\leq T^{p-1} \int_{0}^{T}
		\Big( \e\left[ |X_{i,j}(s)-X_{i,j}(\kappa_n(s))|^{3p}\right] \Big)^{\frac{1}{3}}	
		\Big( \e\left[|X_{i,j}(s)|^{-3p} \right] \Big)^{1/3} \\
		&\qquad \times	\Big( \e\left[|X_{i,j}(\kappa_n(s))|^{-3p}\right]\Big)^{1/3} \mathrm{d}s.
	\end{align*}
	This estimate together with  Hypothesis \ref{Ip} implies 
	\begin{align*}
		&\e\left[\left(\int_{0}^{t_{m}^{(n)}} \left|\frac{1}{X_{i,j}(s)}-\frac{1}{X_{i,j}(\kappa_n(s))} \right| \mathrm{d}s\right)^{p}\right] \leq \frac{C}{n^{p/2}},
	\end{align*}
	for some constant $C>0$.
	Since each $b_i$ is Lipschitz continuous for $i=1,\ldots,d$, by using Hypothesis \ref{Ip}, we have
	\begin{align*}
		&\e\left[\left(\int_{0}^{t_{m}^{(n)}} \left|b_i(X(s))-b_i(X(\eta_n(s)))\right| \mathrm{d}s \right)^p\right]\\
		&\quad	\leq T^{p-1} \|b\|_{Lip}^p \int_{0}^{T} \e\left[ \left|X(s)-X(\eta_n(s)) \right|^p \right]\mathrm{d}s
		\leq \frac{C}{n^{p/2}},
	\end{align*}
	for some constant $C>0$.
	It then follows from \eqref{est_ril} that 
	\begin{align*}
		\e\left[
		\left( \sum_{\ell=0}^{m-1} \sum_{i=1}^{d} |r_i(\ell)| \right)^{p}
		\right]
		\leq \frac{C}{n^{p/2}},
	\end{align*}
	for some constant $C>0$.
	This estimate together with \eqref{est_sup_e} yields the desired result.
\end{proof}

Now we prove that if the drift coefficients $b_i$ are smooth, then the semi-implicit Euler-Maruyama scheme converges at the strong rate of order $1$.
\begin{Thm}\label{main_2}
	Suppose that the assumptions of Theorem \ref{pathwise1} hold. Moreover, suppose that Hypothesis \ref{Ip} holds for some $\hat{p} = 4p\geq 8$ and  $b_{i} \in C_b^2(\real^d;\real)$.
	Then there exists $C>0$ which depends on $d$ such that, for any $n \in \n$ with $T/n \leq 1$,
	\begin{align*}
		\e\left[\sup_{k=1,\ldots,n}|X(t_{k}^{(n)})-X^{(n)}(t_{k}^{(n)})|^p\right]^{1/p}
		\leq \frac{C}{n}.
	\end{align*}
\end{Thm}
\begin{proof}
	For $\x =(x_1,\ldots, x_d) \in \Delta_d$, we  denote
	\begin{align*}
		f_{i}(\x)
		:=\sum_{j \neq i} \frac{\gamma_{i,j}}{x_i-x_j}.
	\end{align*}
	The first and second order derivatives of $f$  are given as follows:
	\begin{align*}
		\partial_m f_i(\x)
		= \frac{\partial f_i}{\partial x_m} :=  
		\left\{ \begin{array}{ll}
			\displaystyle -\sum_{j \neq i} \frac{\gamma_{i,j}}{(x_i-x_j)^2} &\text{ if } m=i,\\
			\displaystyle \frac{\gamma_{i,m}}{(x_i-x_m)^2}  &\text{ if } m \neq i,
		\end{array}\right.
	\end{align*}
	and
	\begin{align*}
		\partial_{\ell} \partial_m f_i(\x)
		:= \frac{\partial^2 f_i}{\partial x_\ell \partial x_m} = 
		\left\{ \begin{array}{ll}
			\displaystyle \sum_{j \neq i} \frac{2 \gamma_{i,j}}{(x_i-x_j)^3} &\text{ if } m=\ell=i,\\
			\displaystyle -\frac{2 \gamma_{i,\ell}}{(x_i-x_\ell)^3} &\text{ if } m=i, \ \ell \neq i,\\
			\displaystyle -\frac{2\gamma_{i,m}}{(x_i-x_m)^3}&\text{ if } m \not = i, \ell = i,\\
			\displaystyle \frac{2\gamma_{i,m}}{(x_i-x_m)^3}  &\text{ if } m \neq i, \ell = m,\\
			\displaystyle 0  &\text{ if } m \neq i, \ell \neq i, m \neq \ell.
		\end{array}\right.
	\end{align*}
	Recall that  for $k=0,\ldots,n-1$, we have
	\begin{align}\label{first_1}
		e_i(t_{k+1}^{(n)})
		=&e_i(t_{k}^{(n)})
		+\left\{
		f_i(X(t_{k+1}^{(n)})) - f_i(X^{(n)}(t_{k+1}^{(n)}))
		\right\} \frac{T}{n}\\
		&\quad 	+\left\{b_i(X(t_{k}^{(n)}))-b_i(X^{(n)}(t_{k}^{(n)}))\right\} \frac{T}{n}
		+r_i(k), \notag
	\end{align}
	and by using It\^o's formula, we have 
	\begin{align*}
		r_{i}(k)=r_{i}^{(1)}(k)+r_{i}^{(2)}(k)+r_{i}^{(3)}(k)+r_{i}^{(4)}(k),
	\end{align*}
	where
	\begin{align*}
		r_{i}^{(1)}(k)
		:=&\int_{t_{k}^{(n)} }^{t_{k+1}^{(n)}} \int_{t}^{t_{k+1}^{(n)}}
		h_{i}^{(1)}(X(s))
		\rd s\rd t, \quad
		r_{i}^{(2)}(k)
		:=\int_{t_{k}^{(n)} }^{t_{k+1}^{(n)}} \int_{t_{k}^{(n)}}^{t}
		h_{i}^{(2)}(X(s))
		\rd s\rd t,\\
		r_{i}^{(3)}(k)
		:=&\sum_{j=1}^{d} \int_{t_{k}^{(n)} }^{t_{k+1}^{(n)}} \int_{t}^{t_{k+1}^{(n)}}
		h_{i,j}^{(3)}(X(s))
		\rd W_j(s) \rd t, \\
		r_{i}^{(4)}(k)
		:=&\sum_{j=1}^{d} \int_{t_{k}^{(n)} }^{t_{k+1}^{(n)}} \int_{t_{k}^{(n)}}^{t}
		h_{i,j}^{(4)}(X(s))
		\rd W_{j}(s)\rd t,
	\end{align*}
	and for $\x=(x_1,\ldots,x_d) \in \Delta_d$,
	\begin{align*}
		h_{i}^{(1)}(\x)
		&:=-\sum_{m=1}^{d} \partial_{m} f_{i} (\x) (f_m(\x) + b_m(\x))
		-\sum_{m,k,k'=1}^d \frac{\sigma_{k,m}\sigma_{k',m}}{2}{\partial_k \partial_{k'}}f_i(\x) \\
		h_{i}^{(2)}(\x)
		&:=\sum_{m=1}^{d} \partial_{m} b_{i} (\x) (f_m(\x) + b_m(\x))
		+\sum_{m,k,k'=1}^d \frac{\sigma_{k,m}\sigma_{k',m}}{2}{\partial_k \partial_{k'}}b_i(\x) \\
		h_{i,j}^{(3)}(\x)
		&=-\sigma_{i,j} \sum_{m=1}^{d} \partial_{m} f_{i} (\x)
		\quad\text{and}\quad
		h_{i,j}^{(4)}(\x)
		=\sigma_{i,j} \sum_{m=1}^{d} \partial_{m} b_{i} (\x).
	\end{align*}
	From \eqref{first_1}, we have
	\begin{align*} 
		&\left|
		e_i(t_{k+1}^{(n)})
		-\left\{
		f_i(X(t_{k+1}^{(n)})) - f_i(X^{(n)}(t_{k+1}^{(n)}))
		\right\} \frac{T}{n}
		\right|^2\\
		&=
		\left|
		e_i(t_{k}^{(n)})
		+\left\{b_i(X(t_{k}^{(n)}))-b_i(X^{(n)}(t_{k}^{(n)}))\right\} \frac{T}{n}
		+r_i(k)
		\right|^2 \notag
	\end{align*}
	and thus
	\begin{align*}
		|e_i(t_{k+1}^{(n)})|^2
		=&
		|e_i(t_{k}^{(n)})|^2
		-\left| \left\{
		f_i(X(t_{k+1}^{(n)})) - f_i(X^{(n)}(t_{k+1}^{(n)}))
		\right\} \frac{T}{n} \right|^2\\
		&+\left|\left\{b_i(X(t_{k}^{(n)}))-b_i(X^{(n)}(t_{k}^{(n)}))\right\} \frac{T}{n}\right|^2
		+r_i(k)^2\\
		&+2 e_i(t_{k+1}^{(n)}) \left\{
		f_i(X(t_{k+1})) - f_i(X^{(n)}(t_{k+1}))
		\right\} \frac{T}{n}\\
		&+2 e_i(t_{k}^{(n)}) \left\{b_i(X(t_{k}^{(n)}))-b_i(X^{(n)}(t_{k}^{(n)}))\right\} \frac{T}{n}
		+2e_i(t_{k}^{(n)}) r_i(k)\\
		&+2\left\{b_i(X(t_{k}^{(n)}))-b_i(X^{(n)}(t_{k}^{(n)}))\right\} \frac{T}{n} r_i(k).
	\end{align*}
	Using the identity \eqref{id1}, the fact that  $e_j-e_i=X_{j,i}-X_{j,i}^{(n)}$, the Lipschitz continuity of $b_i$, the inequality $xy\leq x^2/2+y^2/2$ and the fact that $T/n \leq 1$, we get
	\begin{align*}
		|e(t_{k+1}^{(n)})|^2
		&\leq
		|e(t_{k}^{(n)})|^2
		+d \|b\|_{Lip}^2 |e(t_k^{(n)})|^2 \left( \frac{T}{n} \right)^2
		+2 d \|b\|_{Lip} |e(t_k^{(n)})|^2 \frac{T}{n}\\
		&\quad
		+2 \sum_{i=1}^{d} e_i(t_{k}^{(n)}) r_i(k)
		+2 \|b\|_{Lip} |e(t_{k}^{(n)})| \frac{T}{n} \sum_{i=1}^{d} r_i(k)
		+\sum_{i=1}^{d} r_i(k)^2\\
		&\leq
		|e(t_{k}^{(n)})|^2
		+C_2 |e(t_k^{(n)})|^2 \frac{T}{n}
		+2 \sum_{i=1}^{d} e_i(t_{k}^{(n)}) r_i(k)
		+\frac{3}{2}\sum_{i=1}^{d} r_i(k)^2,
	\end{align*}
	where $C_2:=d\{3\|b\|_{Lip}^2+2\|b\|_{Lip}\}$.
	Thus, we obtain
	\begin{align*}
		|e(t_{k}^{(n)})|^2
		\leq \sum_{j=0}^{k-1} \left\{
		C_2 |e(t_j^{(n)})|^2 \frac{T}{n}
		+2 \sum_{i=1}^{d} e_i(t_{j}^{(n)}) r_i(j)
		+\frac32 \sum_{i=1}^{d} r_i(j)^2
		\right\}.
	\end{align*}
	Hence for $p=2q\geq 2$, we have
	\begin{align}\label{first_1_1}
		\sup_{k=0,\ldots,\ell}|e(t_{k}^{(n)})|^{2q}
		&\leq 3^{q-1}C_2^{q} \sup_{k=0,\ldots,\ell} \left|
		\sum_{j=0}^{k-1} |e(t_j^{(n)})|^2 \frac{T}{n}
		\right|^{q} \\
		&\quad+ 3^{q-1}2^q \sup_{k=0,\ldots,\ell} \left|
		\sum_{j=0}^{k-1} \sum_{i=1}^{d} e_i(t_{j}^{(n)}) r_i(j)
		\right|^{q} \notag\\
		&\quad+ 3^{2q-1}2^{-q} \sup_{k=0,\ldots,\ell}
		\left|
		\sum_{j=0}^{k-1} \sum_{i=1}^{d} r_i(j)^2
		\right|^{q}. \notag
	\end{align}
	
	Now for $r \in [1,2q]$, we consider the upper bound of the $r$-th moment of $r_i^{(1)}(j), r_i^{(2)}(j), r_i^{(3)}(j)$ and $r_i^{(4)}(j)$.
	Using Jensen's inequality and the inequality $xy\leq x^2/2+y^2/2$, there exist $K_r^{(1)}, K_r^{(2)}, K_r^{(3)}$ and $K_r^{(4)}$ such that
	\begin{align*}
		&|h_{i}^{(1)}(\x)|^{r}	\leq (2d+d^3)^{r-1} \left\{ \sum_{m=1}^{d}|\partial_{m} f_{i} (\x)|^r (|f_m(\x)|^r + |b_m(\x)|^r) \right.\\
		&\qquad \left.	+\sum_{m,k,k'=1}^{d} |\sigma_{k,m}|^r|\sigma_{k',m}|^r|\partial_{k} \partial_{k'} f_i(\x)|^r \right\}\\
		&\leq
		(2d+d^3)^{r-1}\left\{ \sum_{m=1}^{d}(|\partial_m f_i(\x)|^{2r}+|f_m(\x)|^{2r}
		+\|b_m\|_{\infty}^{r}|\partial_{m}f_i(\x)|^{r}) \right.\\	
		&	\qquad + \left. \sum_{m,k,k'=1}^{d} |\sigma_{k,m}|^r|\sigma_{k',m}|^r|\partial_k \partial_{k'} f_i(\x)|^r \right\}\\
		&\leq K_r^{(1)}
		\left\{
		\sum_{j \neq i} \left|\frac{1}{x_i-x_j}\right|^{4r}
		+\sum_{j \neq i} \left|\frac{1}{x_i-x_j}\right|^{2r}
		+\sum_{j \neq i} \left|\frac{1}{x_i-x_j}\right|^{3r}
		\right\},
	\end{align*}
	and
	\begin{align*}
		&|h_i^{(2)}(\x)|^{r}\\
		&\leq (d+d^3)^{r-1}
		\left\{
		2^{r-1}
		\sum_{m=1}^{d}
		\|\partial_{m}b_i\|_{\infty}^{r} (|f_i(\x)|^r + \|b_i\|_{\infty}^{r})
		+\sum_{m,k,k'=1}^{d} \left|\frac{\sigma_{k,m} \sigma_{k',m}}{2}\right|^r
		\|\partial_k \partial_{k'} b_i\|_{\infty}^r
		\right\}\\
		&\leq K_r^{(2)}
		\left\{
		\sum_{m=1}^{d} \sum_{k \neq m} \left|\frac{1}{x_k-x_m}\right|^{r}
		+1
		\right\},
	\end{align*}
	and
	\begin{align*}
		&|h_{i,j}^{(3)}(\x)|^{r}
		\leq d^{r-1} |\sigma_{i,j}|^r \sum_{m=1}^{d} |\partial_m f_i(\x)|^{r}
		\leq K_r^{(3)} \sum_{j \neq i}
		\left|\frac{1}{x_i-x_j}\right|^{2r},
	\end{align*}
	and
	\begin{align*}
		|h_{i,j}^{(4)}(\x)|^{r}
		\leq d^{r-1} |\sigma_{i,j}|^r \sum_{m=1}^{d} \|\partial_{m} b_{i}\|_{\infty}^r
		\leq K_r^{(4)}.
	\end{align*}
	Thus, from Hypothesis \ref{Ip}, there exist $K_r^{(1,2)}$ and $K_r^{(3,4)}$ such that
	\begin{align*}
		&\e[|r_i^{(1)}(j)|^{r}]+\e[|r_i^{(2)}(j)|^{r}] \notag\\
		&\leq \left(\frac{T}{n}\right)^{2(r-1)}
		\int_{t_{k}^{(n)}}^{t_{k+1}^{(n)}} \rd t \int_{t_{k}^{(n)}}^{t_{k+1}^{(n)}}\rd s 
		\e\left[
		|h_{i}^{(1)}(X(s))|^{r}
		+|h_{i}^{(2)}(X(s))|^{r}
		\right]
		\leq K_r^{(1,2)} \left( \frac{T}{n} \right)^{2r}
	\end{align*}
	and by using Burkholder-Davis-Gundy's inequality,
	\begin{align*}
		&\e[|r_i^{(3)}(j)|^{r}]+\e[|r_i^{(4)}(j)|^{r}]\\	
		&\leq d^{r-1}\sum_{j=1}^{d} \left(\frac{T}{n}\right)^{r-1}
		\int_{t_{k}^{(n)} }^{t_{k+1}^{(n)}}
		\e\left[
		\left|
		\int_{t}^{t_{k+1}^{(n)}}
		h_{i,j}^{(3)}(X(s))
		\rd W_j(s) 
		\right|^r
		+
		\left|
		\int_{t_{k}^{(n)}}^{t}
		h_{i,j}^{(4)}(X(s))
		\rd W_j(s) 
		\right|^r
		\right]	
		\rd t \notag\\
		&\leq
		c_r d^{r-1}\sum_{j=1}^{d} \sum_{m'=3}^{4} \left(\frac{T}{n}\right)^{\frac{3r}{2}-2}
		\int_{t_{k}^{(n)} }^{t_{k+1}^{(n)}} \rd t \int_{t_{k}^{(n)} }^{t_{k+1}^{(n)}} \rd s
		\e\left[
		|h_{i,j}^{(m')}(X(s))|^{r}
		\right] \notag\\
		&\leq K_r^{(3,4)} \left( \frac{T}{n} \right)^{\frac{3r}{2}}. \notag
	\end{align*}
	
	Let $M_{k}:=\sum_{j=0}^{k-1} \sum_{i=1}^{d} e_i(t_{j}^{(n)}) \{r_i^{(3)}(j) + r_i^{(4)}(j)\}$.
	Then it follows from Hypothesis \ref{Ip} and the upper bound of $h_{i,m}^{(3)}(x)$ and $h_{i,j}^{(4)}(x)$ that
	\begin{align*}
		\e\left[M_k| \mathcal{F}_{t_{k-1}^{(n)}} \right]
		=M_{k-1}+\sum_{i=1}^{d} e_{i}(t_{k-1}) \e\left[r_i^{(3)}(k-1) + r_i^{(4)}(k-1) \Big| \mathcal{F}_{t_{k-1}^{(n)}} \right]
		=M_{k-1}.
	\end{align*}
	Hence $(M_k)_{k=1,\ldots,n}$ is a $(\mathcal{F}_{t_{k}^{(n)}})_{k=1,\ldots,n}$-martingale.
	By using Burkholder-Davis-Gundy's inequality, we have
	\begin{align*}
		\e\left[ \sup_{k=0,\ldots,\ell} |M_k|^q\right]
		&\leq c_q \e\left[ \left\{ \sum_{j=0}^{\ell-1} \sum_{i=1}^{d} |e_i(t_{j}^{(n)})|^2 |r_i^{(3)}(j) + r_i^{(4)}(j)|^2 \right\}^{q/2}\right] \notag\\
		&\leq 2^{q-1} d^{\frac{q}{2}-1} c_q \sum_{j=0}^{\ell-1} \sum_{i=1}^{d}  n^{\frac{q}{2}-1} \e\left[ |e_i(t_{j}^{(n)})|^{q} \{|r_i^{(3)}(j)|^q + |r_i^{(4)}(j)|^{q}\}\right].
	\end{align*}
	Therefore, by taking the expectation of \eqref{first_1_1}, we obtain
	\begin{align*}
		&\e\left[\sup_{k=0,\ldots,\ell}|e(t_{k}^{(n)})|^{2q}\right]\\
		&\leq 3^{q-1}C_2^{q} \e\left[\sup_{k=0,\ldots,\ell} \left|
		\sum_{j=0}^{k-1} |e(t_j^{(n)})|^2 \frac{T}{n}
		\right|^{q}\right]\\
		&\quad+ 3^{q-1}2^{2q-1} \e\left[\sup_{k=0,\ldots,\ell} \left|
		\sum_{j=0}^{k-1} \sum_{i=1}^{d} e_i(t_{j}^{(n)}) \{r_i^{(1)}(j)+r_i^{(2)}(j)\}
		\right|^{q} \right]
		+3^{q-1}2^{2q-1} \e\left[\sup_{k=0,\ldots,\ell} \left| M_k
		\right|^{q} \right]\\
		&\quad+ 3^{2q-1}2^{-q} \e\left[\sup_{k=0,\ldots,\ell}
		\left|
		\sum_{j=0}^{k-1} \sum_{i=1}^{d} r_i(j)^2
		\right|^{q}\right]\\
		&\leq
		3^{q-1}C_2^{q} \e\left[\sup_{k=0,\ldots,\ell} \left|
		\sum_{j=0}^{k-1} |e(t_j^{(n)})|^2 \frac{T}{n}
		\right|^{q}\right]\\
		&\quad+ 3^{q-1}2^{3q-2} d^{q-1} \sum_{j=0}^{\ell-1} \sum_{i=1}^{d} \sum_{m=1}^2 n^{q-1}
		\e\left[\left|e_i(t_{j}^{(n)})\right|^{q} |r_i^{(m)}(j)|^q \right]\\
		&\quad 	+3^{q-1}2^{3q-2} d^{\frac{q}{2}-1} c_q \sum_{j=0}^{\ell-1} \sum_{i=1}^{d}   \sum_{m=3}^4  n^{\frac{q}{2}-1} \e\left[ |e_i(t_{j}^{(n)})|^{q} |r_i^{(m)}(j)|^q \right]\\
		&\quad+ 3^{2q-1}2^{-q} d^{q-1} \sum_{j=0}^{\ell-1} \sum_{i=1}^{d} n^{q-1}\e\left[\left|r_i(j)\right|^{2q}\right]\\
		&=3^{q-1}C_2^{q} \e\left[\sup_{k=0,\ldots,\ell} \left|
		\sum_{j=0}^{k-1} |e(t_j^{(n)})|^2 \frac{T}{n}
		\right|^{q}\right]
		+I_{\ell}^{(1,2)}+I_{\ell}^{(3,4)}+J_{\ell}.
	\end{align*}
	From H\"older's inequality and the inequality $xy\leq x^2/2+y^2/2$, we have
	\begin{align*}
		I_{\ell}^{(1,2)}
		&\leq 3^{q-1}2^{3q-2} d^{q-1} \sqrt{K_{2q}^{(1,2)}} \sum_{j=0}^{\ell-1} \sum_{i=1}^{d} n^{q-1}
		\left(\e\left[\left|e_i(t_{j}^{(n)})\right|^{2q}\right]\right)^{1/2} \left( \frac{T}{n}\right)^{2q}\\
		&=(3T)^{q-1}2^{3q-2}d^{q-1} \sqrt{K_{2q}^{(1,2)}} \sum_{j=0}^{\ell-1} \sum_{i=1}^{d}
		\left(\e\left[\left|e_i(t_{j}^{(n)})\right|^{2q}\right]\right)^{1/2} \left( \frac{T}{n}\right)^{q+1}\\
		&\leq (3T)^{q-1}2^{3q-3} d^{q-1} \sqrt{K_{2q}^{(1,2)}} \sum_{j=0}^{\ell-1} \sum_{i=1}^{d} 
		\left\{
		\e\left[\left|e_i(t_{j}^{(n)})\right|^{2q}\right] \frac{T}{n}+\left( \frac{T}{n}\right)^{2q+1}
		\right\}\\
		&\leq \widetilde{K}_{2q}^{(1,2)}
		\left\{
		\sum_{j=0}^{\ell-1} \sum_{i=1}^{d} 
		\e\left[\left|e_i(t_{j}^{(n)})\right|^{2q}\right] \frac{T}{n}
		+\frac{1}{n^{2q}}
		\right\},
	\end{align*}
	for some constant $\widetilde{K}_{2q}^{(1,2)}$ and
	\begin{align*}
		I_{\ell}^{(3,4)}
		&\leq 3^{q-1}2^{3q-2} d^{\frac{q}{2}-1} c_q \sqrt{K_{2q}^{(3,4)}}\sum_{j=0}^{\ell-1} \sum_{i=1}^{d}   n^{\frac{q}{2}-1} \left(\e \left[ |e_i(t_{j}^{(n)})|^{2q} \right]\right)^{1/2} \left(\frac{T}{n}\right)^{3q/2}\\
		&=3^{q-1}2^{3q-2} (dT)^{\frac{q}{2}-1} c_q \sqrt{K_{2q}^{(3,4)}}\sum_{j=0}^{\ell-1} \sum_{i=1}^{d}\left( \e\left[ |e_i(t_{j}^{(n)})|^{2q} \right]\right)^{1/2} \left(\frac{T}{n}\right)^{q+1} \\
		&\leq 3^{q-1}2^{3q-3} d^{\frac{q}{2}-1} c_q \sqrt{K_{2q}^{(3,4)}} \sum_{j=0}^{\ell-1} \sum_{i=1}^{d} 
		\left\{
		\e\left[\left|e_i(t_{j}^{(n)})\right|^{2q}\right] \frac{T}{n}+\left( \frac{T}{n}\right)^{2q+1}
		\right\}\\
		&\leq \widetilde{K}_{2q}^{(3,4)}
		\left\{
		\sum_{j=0}^{\ell-1} \sum_{i=1}^{d} 
		\e\left[\left|e_i(t_{j}^{(n)})\right|^{2q}\right] \frac{T}{n}
		+\frac{1}{n^{2q}}
		\right\},
	\end{align*}
	for some constant $\widetilde{K}_{2q}^{(3,4)}$.
	Finally, we have
	\begin{align*}
		J_{\ell}
		&\leq 3^{2q-1} 2^{3q-2} d^{q-1} \sum_{j=0}^{\ell-1} \sum_{i=1}^{d} \sum_{m=1}^{4} n^{q-1}\e\left[\left|r_i^{(m)}(j)\right|^{2q}\right] \\
		&\leq 3^{2q-1} 2^{3q-2} (dT)^{q-1} \sum_{j=0}^{\ell-1} \sum_{i=1}^{d} \left\{ \widetilde{K}_{2q}^{(1,2)} \left(\frac{T}{n}\right)^{3q+1}+\widetilde{K}_{2q}^{(3,4)}\left(\frac{T}{n}\right)^{2q+1}\right\}\\
		&\leq \frac{\widetilde{K}_{2q}}{n^{2q}},
	\end{align*}
	for some $\widetilde{K}_{2q}$.
	
	Therefore, we obtain for some constant $C>0$ that 
	\begin{align*}
		\e\left[\sup_{k=0,\ldots,\ell}|e(t_{k}^{(n)})|^{2q}\right]
		\leq C \sum_{j=0}^{\ell-1} \e\left[|e(t_{j}^{(n)})|^{2q}\right] \frac{T}{n}
		+\frac{C}{n^{2q}}.
	\end{align*}
	By Gronwall's inequality, we conclude the proof.	
\end{proof}

\subsection{The case of general diffusion coefficient}

In this section, we develop the argument presented in the previous sections to establish the convergence in $L^2$-norm of the semi-implicit Euler-Maruyama scheme for equation \eqref{def:X} in the case that the diffustion coefficient $\sigma$ may depend on $X$. 


\begin{Thm}\label{main_3}
	Suppose that Hypothesis \ref{Ip} holds for $\widehat{p} = 6$.
	Then there exists $C>0$ which depends on $d$ such that for any $n \in \n$ with $T/n \leq 1$,
	\begin{align}\label{main_4_0}
		\sup_{k=1,\ldots,n}
		\e\left[\left|X(t_{k}^{(n)})-X^{(n)}(t_{k}^{(n)}) \right|^2\right]^{1/2}
		\leq \frac{C}{n^{1/2}},
	\end{align}
	and
	\begin{align}\label{main_4_1}
		\e\left[
		\sup_{k=1,\ldots,n}
		\left|
		X(t_{k}^{(n)})
		-X^{(n)}(t_{k}^{(n)})
		\right|^2
		\right]^{1/2}
		\leq \frac{C}{n^{1/4}}.
	\end{align}	
\end{Thm}


\begin{proof}
	We first recall that $f_{i}(\x):=\sum_{j \neq i} \frac{\gamma_{i,j}}{x_i-x_j}$.
	It follows from \eqref{def_ei} that
	\begin{align*}
		&\left|
		e_i(t_{k+1}^{(n)})
		-\left\{
		f_i(X(t_{k+1}^{(n)})) - f_i(X^{(n)}(t_{k+1}^{(n)}))
		\right\}
		\frac{T}{n}
		\right|^2
		=
		\left|
		e_i(t_{k}^{(n)})
		+R_i(k)
		\right|^2,
	\end{align*}
	where $R_i(k):=R_{b,i}(k)+R_{\sigma,i}(k)+r_i(k)$ and
	\begin{align*}
		R_{b,i}(k)
		&:=
		\left\{
		b_i(X(t_{k}^{(n)}))-b_i(X^{(n)}(t_{k}^{(n)}))
		\right\}
		\frac{T}{n},\\
		R_{\sigma,i}(k)
		&:=\sum_{j=1}^{d}
		\left\{
		\sigma_{i,j}(X(t^{(n)}_{k}))
		-\sigma_{i,j}(X^{(n)}(t^{(n)}_{k}))
		\right\}
		\{W_j(t_{k+1}^{(n)})-W_j(t_{k}^{(n)})\},
	\end{align*}
	and $r_i(k)$ is defined by \eqref{def_rik}. Thus we have
	\begin{align*}
		|e_i(t_{k+1}^{(n)})|^2
		=&
		|e_i(t_{k}^{(n)})|^2
		-\left| \left\{
		f_i(X(t_{k+1}^{(n)})) - f_i(X^{(n)}(t_{k+1}^{(n)}))
		\right\} \frac{T}{n} \right|^2\\
		&+2 e_i(t_{k+1}^{(n)}) \left\{
		f_i(X(t_{k+1})) - f_i(X^{(n)}(t_{k+1}))
		\right\} \frac{T}{n}\\
		&+2e_i(t_{k}^{(n)}) R_i(k)
		+R_i(k)^2.
	\end{align*}
	Using the identity \eqref{id1}, the fact that  $e_j-e_i=X_{j,i}-X_{j,i}^{(n)}$, the Lipschitz continuity of $b_i$ and $T/n \leq 1$, we get
	\begin{align}\label{main3_1}
		&|e(t_{k+1}^{(n)})|^2
		\leq 
		|e(t_{k}^{(n)})|^2
		+2\sum_{i=1}^{d} e_i(t_{k}^{(n)}) R_i(k)
		+\sum_{i=1}^{d} R_i(k)^2 \\
		&\leq
		|e(t_{k}^{(n)})|^2
		+|e(t_{k}^{(n)})| \sum_{i=1}^{d} |e_i(t_{k}^{(n)})| \frac{2\|b\|_{Lip} T}{n}
		+2\sum_{i=1}^{d} e_i(t_{k}^{(n)}) R_{\sigma,i}(k)
		+2\sum_{i=1}^{d} e_i(t_{k}^{(n)}) r_{i}(k)  \notag\\
		&\quad
		+|e(t_{k}^{(n)})|^2 \frac{3d\|b\|_{Lip}^2T^2}{n^2}
		+3\sum_{i=1}^{d} R_{\sigma,i}(k)^2
		+3\sum_{i=1}^{d} r_{i}(k)^2  \notag\\
		&\leq
		|e(t_{k}^{(n)})|^2
		+C_3 |e(t_{k}^{(n)})|^2 \frac{T}{n}
		+2\sum_{i=1}^{d} e_i(t_{k}^{(n)}) R_{\sigma,i}(k)
		+2\sum_{i=1}^{d} e_i(t_{k}^{(n)}) r_{i}(k)  \notag\\
		&\quad+3\sum_{i=1}^{d} R_{\sigma,i}(k)^2
		+3\sum_{i=1}^{d} r_{i}(k)^2,  \notag
	\end{align}
	where $C_3:=2d\|b\|_{Lip}+3d\|b\|_{Lip}^2$.
	Because of the independent increment property of Brownian motion $W$, the expectation of $e_i(t_{k}^{(n)}) R_{\sigma,i}(k)$ equals to zero.
	Therefore, by taking the expectation in \eqref{main3_1} and using the Lipschitz continuity of $\sigma_{i,j}$, we obtain
	\begin{align*}
		\e\left[
		|e(t_{k+1}^{(n)})|^2
		\right]
		&\leq
		\e\left[
		|e(t_{k}^{(n)})|^2
		\right]
		+C_4 \e\left[
		|e(t_{k}^{(n)})|^2
		\right]
		\frac{T}{n}\\
		&\quad+2\sum_{i=1}^{d}
		\e\left[
		e_i(t_{k}^{(n)}) r_{i}(k)
		\right]
		+3
		\e\left[
		|r(k)|^2
		\right],
	\end{align*}
	where $C_4:=C_3+d^3 \|\sigma\|_{Lip}^2$.
	Thus we have for any $k=1,\ldots,n$,
	\begin{align}
		\label{pr_5}
		&\e\left[
		|e(t_{k}^{(n)})|^2
		\right]\\
		&\leq
		C_4 
		\sum_{\ell=0}^{k-1}
		\e\left[
		|e(t_{\ell}^{(n)})|^2
		\right]
		\frac{T}{n}
		+2\sum_{\ell=0}^{k-1}
		\sum_{i=1}^{d}
		\e\left[
		e_i(t_{\ell}^{(n)}) r_{i}(\ell)
		\right]
		+3\sum_{\ell=0}^{k-1}
		\e\left[
		|r(\ell)|^2
		\right]. \notag
	\end{align}
	Recall that $r_i(k)=r_{f,i}(k)+r_{b,i}(k)+r_{\sigma,i}(k)$, where
	\begin{align*}
		r_{f,i}(k)
		&:=
		\int_{t_{k}^{(n)}}^{t_{k+1}^{(n)}}
		\left\{
		f_i(X(s))-f_i(X(t_{k}^{(n)}))
		\right\}
		\mathrm{d}s, \\
		r_{b,i}(k)
		&:=\int_{t_{k}^{(n)}}^{t_{k+1}^{(n)}}
		\left\{
		b_i(X(s))-b_i(X(t_{k}^{(n)}))
		\right\}
		\mathrm{d}s, \\
		r_{\sigma,i}(k)
		&:=
		\sum_{j=1}^{d}
		\int_{t_{k}^{(n)}}^{t_{k+1}^{(n)}}
		\left\{
		\sigma_{i,j}(X(s)) -\sigma_{i,j}(X(t_{k}^{(n)}))
		\right\}
		\mathrm{d}W_j(s).
	\end{align*}
	
	We now estimate the expectation of $|r(\ell)|^2$.
	By using H\"older's inequality and Hypothesis \ref{Ip} with $\widehat{p}=6$, we have
	\begin{align}\label{pr_6}
		&\e\left[|r_{f}(\ell)|^2\right]
		=\sum_{i=1}^{d}\e\left[|r_{f,i}(\ell)|^2\right]\\
		&\leq \sum_{i=1}^{d} \frac{T}{n}
		\int_{t_{\ell}^{(n)}}^{t_{\ell+1}^{(n)}}
		\e\left[
		\left|
		f_i(X(s))
		-f_i(X(t_{\ell+1}^{(n)}))
		\right|^2
		\right]
		\mathrm{d}s \notag\\
		&\leq \sum_{i=1}^{d} \sum_{j \neq i}
		\frac{(d-1) \gamma_{i,j}^2 T}{n}
		\int_{t_{\ell}^{(n)}}^{t_{\ell+1}^{(n)}}
		\e\left[
		\frac{
			\left|X_{i,j}(s) -X_{i,j}(t_{\ell+1}^{(n)}) \right|^2
		}
		{
			\left|X_{i,j}(s)\right|^2\left|X_{i,j}(t_{\ell+1}^{(n)}) \right|^2
		}
		\right] \mathrm{d}s \notag\\
		&\leq \sum_{i=1}^{d} \sum_{j \neq i}
		\frac{(d-1) \gamma_{i,j}^2 T}{n}  \notag\\
		&\times \int_{t_{\ell}^{(n)}}^{t_{\ell+1}^{(n)}}
		\e\left[\left|X_{i,j}(s) -X_{i,j}(t_{\ell+1}^{(n)}) \right|^6\right]^{\frac{1}{3}}
		\e\left[ \left|X_{i,j}(s)\right|^{-6} \right]^{\frac{1}{3}}
		\e\left[ \left|X_{i,j}(t_{\ell+1}^{(n)}) \right|^{-6} \right]^{\frac{1}{3}}
		\mathrm{d}s \notag\\
		&\leq
		\sum_{i=1}^{d} \sum_{j \neq i}
		4\hat{C}(d-1) \gamma_{i,j}^2
		\left(
		\frac{T}{n}
		\right)^{3}
		=C_{f}
		\left(
		\frac{T}{n}
		\right)^{3}
		\notag.
	\end{align}
	From the Lipschitz continuity of $b_i$ for each $i=1,\ldots,d$ and Jensen's inequality, we have
	\begin{align}\label{pr_8}
		\e\left[
		|r_{b}(\ell) |^2
		\right]
		&=\sum_{i=1}^{d}
		\e\left[
		|r_{b,i}(\ell)|^2
		\right]\\
		&\leq \frac{d \|b\|_{Lip}^2 T}{n}
		\int_{t_{\ell}^{(n)}}^{t_{\ell+1}^{(n)}}
		\e\left[
		|X(s) - X(t_{\ell}^{(n)})|^2
		\right]
		\mathrm{d}s \notag\\
		& \leq d\widehat{C}^{1/3} \|b\|_{Lip}^2
		\left(
		\frac{T}{n}
		\right)^{3}
		=C_b
		\left(
		\frac{T}{n}
		\right)^{3}
		\notag.
	\end{align}
	From Burkholder-Davis-Gundy's inequality and the Lipschitz continuity of $\sigma_{i,j}$, there exists $c_2>0$ such that
	\begin{align}\label{pr_9}
		\e\left[
		\left|
		r_{\sigma}(\ell)
		\right|^2
		\right]
		&=\sum_{i=1}^{d}
		\e\left[
		\left|
		r_{\sigma,i}(\ell)
		\right|^2
		\right] \\
		&\leq d^3 c_2 \|\sigma\|_{Lip}^2
		\int_{t_{\ell}^{(n)}}^{t_{\ell+1}^{(n)}}
		\e\left[
		|X(s) - X(t_{\ell}^{(n)})|^2
		\right]
		\mathrm{d}s\notag\\
		&\leq
		d^3 c_2 \hat{C}^{1/3} \|\sigma\|_{Lip}^2
		\left(
		\frac{T}{n}
		\right)^2
		=
		C_{\sigma}
		\left(
		\frac{T}{n}
		\right)^{2}. \notag
	\end{align}
	
	Next we consider $\sum_{i=1}^{d}\e[e_i(t_{\ell}^{(n)}) r_{i}(\ell)]$.
	Since $e_i(t_{\ell}^{(n)})$ is $\mathcal{F}_{t_{\ell}^{(n)}}$-measurable and the conditional expectation $\e[ r_{\sigma,i}(\ell) ~|~\mathcal{F}_{t_{\ell}^{(n)}} ]$ equals to zero for each $i=1,\ldots,n$, we obtain
	\begin{align*}
		\sum_{i=1}^{d}
		\e[
		e_i(t_{\ell}^{(n)}) r_{\sigma, i}(\ell)
		]
		=\sum_{i=1}^{d}
		\e\left[
		e_i (t_{\ell}^{(n)})
		\e\left[
		r_{\sigma,i}(\ell)
		~\Big|~
		\mathcal{F}_{t_{\ell}^{(n)}}
		\right]
		\right]
		=0.
	\end{align*}
	Hence, from \eqref{pr_6} and \eqref{pr_8}  and the  inequality $xy \leq x^2/2 + y^2/2$, we have
	\begin{align}
		\label{pr_10}
		\sum_{i=1}^{d}
		\e[
		e_i(t_{\ell}^{(n)})
		r_{i}(\ell)
		]
		&=\sum_{i=1}^{d}
		\e[
		e_i(t_{\ell}^{(n)})
		(
		r_{f, i}(\ell)
		+r_{b, i}(\ell)
		)
		]\\
		&\leq
		\frac{1}{2}
		\e\left[
		|e_i(t_{\ell}^{(n)})|^2
		\right]
		\frac{T}{n}
		+\frac{1}{2} \sum_{i=1}^{d}
		\e\left[
		|r_{f, i}(\ell)+r_{b, i}(\ell)|^2
		\right]
		\frac{n}{T} \notag\\
		&\leq
		\frac{1}{2}\e\left[|e_i(t_{\ell}^{(n)})|^2\right] \frac{T}{n}
		+(C_f+C_b)
		\left(
		\frac{T}{n}
		\right)^{2}. \notag
	\end{align}

	Therefore, it follows from \eqref{pr_5},  \eqref{pr_6}, \eqref{pr_8}, \eqref{pr_9}, \eqref{pr_10} and the fact $T/n\leq 1$ that, for each $k=1,\ldots n$,
	\begin{align*}
		&\e\left[
		|e(t_{k}^{(n)})|^2
		\right]\\
		&\leq
		(C_4+1)
		\sum_{\ell=0}^{k-1}
		\e\left[
		|e(t_{\ell}^{(n)})|^2
		\right]
		\frac{T}{n}
		+\frac{2(C_f+C_b)T^2}{n}
		+\frac{6(C_f+C_b+C_{\sigma})T^2}{n}.
	\end{align*}
	Using discrete type Gronwall's inequality (e.g. Chapter XIV, Theorem 1 and Remark 1,2 in \cite{MiPeFi}, page 436-437), we obtain \eqref{main_4_0}.

	

	Now we prove \eqref{main_4_1}. 
	It follows from \eqref{main3_1}, the Lipschitz continuity of $\sigma_{i,j}$ and Schwarz's inequality that
	\begin{align*}
		\sup_{k=1,\ldots,n}|e(t_{k}^{(n)})|^2
		&\leq
		C_3 \sum_{\ell=1}^{n-1}
		|e(t_{\ell}^{(n)})|^2 \frac{T}{n} \\
		&\quad
		+2\|\sigma\|_{Lip}
		\sum_{\ell=1}^{n-1}
		\sum_{j=1}^{d}
		|e(t_{\ell}^{(n)})|^2
		|W_j(t_{\ell+1}^{(n)})-W_j(t_{\ell}^{(n)})|\notag\\
		&\quad
		+3\|\sigma\|_{Lip}^2\sum_{\ell=1}^{n-1}
		\sum_{j=1}^{d}
		|e(t_{\ell}^{(n)})|^2
		|W_j(t_{\ell+1}^{(n)})-W_j(t_{\ell}^{(n)})|^2 \notag\\
		&\quad
		+2\sum_{\ell=1}^{n-1}
		|e(t_{\ell}^{(n)})| |r(\ell)|
		+3\sum_{\ell=1}^{n-1}
		|r(\ell)|^2.  \notag
	\end{align*}
	Since the random variables $e(t_{\ell}^{(n)})$ and $W_j(t_{\ell+1}^{(n)})-W_j(t_{\ell}^{(n)})$ are independent, by taking the expectation and by using H\"older inequality, $T/n\leq 1$, \eqref{pr_6}, \eqref{pr_8} and \eqref{pr_9}, we have
	\begin{align*}
		&\e\left[
		\sup_{k=1,\ldots,n}|e(t_{k}^{(n)})|^2
		\right] \label{sup_pr3}\\
		&\leq
		\left\{
		C_3T
		+2d \|\sigma\|_{Lip} (nT)^{1/2}
		+3d \|\sigma\|_{Lip}^2 \frac{T}{n}
		\right\}
		\sup_{k=1,\ldots,n}
		\e\left[
		|e(t_{k}^{(n)})|^2
		\right]
		\notag\\
		&\quad
		+2\cdot 3^{1/2} (C_f+C_b+C_{\sigma})^{1/2}T
		\sup_{k=1,\ldots,n}
		\e\left[
		|e(t_{k}^{(n)})|^2
		\right]^{1/2} \notag\\
		\notag
		&\quad
		+\frac{9(C_f+C_b+C_{\sigma})T^2}{n}.
	\end{align*}
	This estimate together with \eqref{main_4_0} implies \eqref{main_4_1}.
\end{proof}


\section{Examples} \label{Sec:moment}
In this section, we will study some classes of SDEs \eqref{def:X} which have a unique non-colliding strong solution satisfying Hypothesis \ref{Ip}.
Note that under Assumptions (A1)-(A4), the coefficients of equation \eqref{def:X} are locally Lipschitz continuous on $\Delta_d$. Therefore, given $X(0)\in \Delta_d$, equation \eqref{def:X} has a unique strong local solution up to the stopping time 
\begin{equation} \label{def:tau}
	\tau=\inf\{t > 0: \min_{1\leq i \leq d-1}|X_{i+1}(t) - X_i(t)| = 0 \text{ or } \max_{1 \leq i \leq d}|X_i(t)| = \infty\}.
\end{equation}
In order to show the existence and uniqueness of global solution to equation \eqref{def:X}, it is sufficient to prove that $\tau = \infty$ almost surely. 

\subsection{Interacting Brownian particles}
We consider the following interacting Brownian particle systems
\begin{align} \label{def:ibp}
	\mathrm{d}X_i(t)
	=\left\{
	\sum_{j \not = i}\frac{\gamma}{X_i(t) - X_j(t)}
	+ b_i(X_i(t))
	\right\}dt
	+ \sum_{j=1}^d \sigma_{i,j}(X(t))\mathrm{d}W_j(t),
	\ i=1,\ldots,d,
\end{align}
with $X(0) \in \Delta_d=\{(x_1,\ldots,x_d)^{*} \in \real^d: x_1 < x_2 <\cdots < x_d\}$.

\begin{Ass}\label{Ass_2}
	Suppose that the domain of the drift coefficient $b$ is $\real$ and it holds that $b_i(x) \leq b_{i+1}(x)$ for any $x \in \real$.
\end{Ass}

These systems contain several classes of well-known particle systems such as the Dyson Brownian particle systems, Dyson-Ornstein-Uhlenbeck process, and the systems considered by  C\'epa and L\'epingle \cite{CepaLepingle}.
Graczyk and Malecki \cite{GrMa14} studied a  class of non-colliding particle systems satisfying condition   $\sigma_{i,j}(x) = \delta_{i,j}\sigma_i(x_i)$, where $\delta_{i,j}$ is the Dirac delta function.
In particular, they obtained the following result.   
\begin{Prop}[\cite{GrMa14}, Corollary 6.2]
	Suppose that 
	\begin{itemize}
		\item  $\sigma_{i,j}(x) = \delta_{i,j}\sigma_i(x_i)$ where $\sigma_i$ be at least $1/2$-H\"older and $\sigma_i^2(x) \leq 2 \gamma$;
		\item $b_i$ be Lipschitz and $b_i(x) \leq b_{i+1}(x)$, $b_i(x) x \leq c(1+|x|^2)$.
	\end{itemize}
	Then the system \eqref{def:ibp} has a unique strong solution in $\Delta_d$ for all $t>0$.
\end{Prop}
In the following we will establish a sufficient condition for the existence and uniqueness of a solution to equation \eqref{def:ibp}. Moreover, we show that Hypothesis \ref{Ip} holds under a certain condition on $\gamma, \sigma$ and $d$.

We need the following elementary inequality. 

\begin{Lem} \label{Lem:aux_1}
	For any $d \geq 2, \ p \geq 0$ and $(x_1,\ldots, x_d) \in \Delta_d$, it holds that
	$$  \sum_{i=1}^{d-1} \sum_{k \not = i, i+1} \frac{1}{(x_{i+1} - x_i)^p(x_{i+1} - x_k)(x_i - x_k)}   
	< \Big(2 - \frac3d\Big)\sum_{i=1}^{d-1} \frac{1}{(x_{i+1} - x_i)^{p+2}} .$$
\end{Lem}

\begin{proof}
	For each $(x_1, \ldots, x_d) \in \Delta_d$, we denote 
	\begin{align*}
		S_1 &= \sum_{i=1}^{d-2} \sum_{k =i+2}^d  \frac{1}{(x_{i+1} - x_i)^p(x_{i+1} - x_k)(x_i - x_k)},\\
		S_2 &= \sum_{i=2}^{d-1} \sum_{k =1}^{i-1}  \frac{1}{(x_{i+1} - x_i)^p(x_{i+1} - x_k)(x_i - x_k)}.
	\end{align*}
	Using Young's inequality
	$$\frac{p}{p+2} a^{p+2} + \frac{1}{p+2} b^{p+2} + \frac{1}{p+2} c^{p+2} \geq a^p b c, \quad a, b,c >0,$$
	we get 
	\begin{align*}
		S_1 &= \sum_{i=1}^{d-2} \sum_{k =i+2}^d  \frac{1}{(k-i-1)(k-i)} \frac{1}{(x_{i+1} - x_i)^p \frac{x_{k} - x_{i+1}}{k-i-1}\frac{x_k - x_i}{k-i}}\\
		&\leq \sum_{i=1}^{d-2} \sum_{k =i+2}^d  \frac{1}{(k-i-1)(k-i)}  \left\{ \frac{p}{p+2} \frac{1}{(x_{i+1} - x_i)^{p+2}} \right. \\
		&\quad + \left. \frac{1}{p+2} \frac{1}{\Big( \frac{x_k - x_{i+1}}{k-i-1}\Big)^{p+2}} 
		+ \frac{1}{p+2} \frac{1}{\Big( \frac{x_k - x_{i}}{k-i}\Big)^{p+2}} \right\}.
	\end{align*}
	Next, using the convexity of the function $a \mapsto a^{-(p+2)}$,  we have the following estimate
	\begin{align} \label{jentsen} \frac{1}{\Big( \frac{a_1+\cdots + a_k}{k} \Big)^{p+2}} \leq \frac{1}{k} \left(\frac{1}{a_1^{p+2}} + \cdots + \frac{1}{a_k^{p+2}}\right), \quad k\geq 1, \  a_1,\ldots, a_k>0. 
	\end{align}
	Since $x_k - x_{i+1} = \sum_{j=i+1}^{k-1}(x_{j+1} - x_j)$ and  $x_k - x_{i} = \sum_{j=i}^{k-1}(x_{j+1} - x_j)$, by applying the inequality \eqref{jentsen}, we get 
	\begin{align*}
		S_1&\leq \sum_{i=1}^{d-2} \sum_{k =i+2}^d  \frac{1}{(k-i-1)(k-i)}  \Bigg\{ \frac{p}{p+2} \frac{1}{(x_{i+1} - x_i)^{p+2}}+ \\
		&\qquad + \frac{1}{p+2} \frac{1}{k-i-1} \sum_{j=i+1}^{k-1} \frac{1}{(x_{j+1}-x_j)^{p+2}}  
		+ \frac{1}{p+2} \frac{1}{k-i} \sum_{j=i}^{k-1} \frac{1}{(x_{j+1}-x_j)^{p+2}}\Bigg\} \\
		&= S_{11} + S_{12},
	\end{align*}
	where 
	\begin{align*}
		S_{11} &= \sum_{i=1}^{d-2}  \sum_{k =i+2}^d \left\{    \frac{p}{p+2}\frac{1}{(k-i-1)(k-i)} + \frac{1}{p+2} \frac{1}{(k-i-1)(k-i)^2}\right\}  \frac{1}{(x_{i+1} - x_i)^{p+2}},\\
		S_{12}& =   \frac{1}{p+2} \sum_{i=1}^{d-2} \sum_{k=i+2}^d \frac{1}{(k-i)(k-i-1)} \Big( \frac{1}{k-i-1} + \frac{1}{k-i}\Big)  \sum_{j=i+1}^{k-1} \frac{1}{(x_{j+1}-x_j)^{p+2}}.
	\end{align*}
	We have 
	\begin{align*}
		S_{11} &= \sum_{i=1}^{d-2}   \left\{    \frac{p}{p+2}\sum_{k =i+2}^d \Big(\frac{1}{k-i-1} - \frac{1}{k-i}\Big) + \frac{1}{p+2} \sum_{k =i+2}^d \frac{1}{(k-i-1)(k-i)^2} \right\}  \frac{1}{(x_{i+1} - x_i)^{p+2}}\\
		& 	= \sum_{i=1}^{d-2}   \left\{    \frac{p}{p+2} \Big(1 - \frac{1}{d-i}\Big) + \frac{1}{p+2} \sum_{k =1}^{d-i-1} \frac{1}{k(k+1)^2}\right\}  \frac{1}{(x_{i+1} - x_i)^{p+2}}
	\end{align*}
	Since $\{(i,j,k) \in \mathbb{N}^3: \ 1 \leq i \leq d-2, \ i+2 \leq k \leq d, \ i+1 \leq j \leq k-1\} = \{(i,j,k) \in \mathbb{N}^3: \ 2 \leq j \leq d-1, \ 1 \leq i \leq j-1, \ j+1 \leq k \leq d\},$ we can rewrite $S_{12}$ as 
	\begin{align*}
		S_{12}
		&= \frac{1}{p+2} \sum_{j=2}^{d-1} \sum_{i=1}^{j-1} \sum_{k=j+1}^d  \frac{1}{(k-i)(k-i-1)} \Big( \frac{1}{k-i-1} + \frac{1}{k-i}\Big)   \frac{1}{(x_{j+1}-x_j)^{p+2}}\\
		&= 	\frac{1}{p+2} \sum_{j=2}^{d-1} \sum_{i=1}^{j-1}  \sum_{k=j+1}^d  \Big( \frac{1}{(k-i-1)^2} - \frac{1}{(k-i)^2}\Big)  \frac{1}{(x_{j+1}-x_j)^{p+2}}\\
		&= 	\frac{1}{p+2} \sum_{j=2}^{d-1} \sum_{i=1}^{j-1}   \Big( \frac{1}{(j-i)^2} - \frac{1}{(d-i)^2}\Big)  \frac{1}{(x_{j+1}-x_j )^{p+2}}\\
		&= 	\frac{1}{p+2} \sum_{j=2}^{d-1}    \Big( \sum_{k=1}^{j-1} \frac{1}{k^2} - \sum_{k=d-j+1}^{d-1} \frac{1}{k^2}\Big)  \frac{1}{(x_{j+1}-x_j )^{p+2}}\\
		&= 	\frac{1}{p+2} \sum_{i=2}^{d-1}    \Big( \sum_{k=1}^{i-1} \frac{1}{k^2} - \sum_{k=d-i+1}^{d-1} \frac{1}{k^2}\Big)  \frac{1}{(x_{i+1}-x_i )^{p+2}},
	\end{align*}	 
	where we replace the index $j$ by $i$ at the last equality. 
	Therefore, 
	\begin{align*}
		S_1&\leq \sum_{i=1}^{d-1}\left\{ \frac{p}{p+2} \left(1-\frac{1}{d-i} \right) +  \right.\\
		&\quad \left.+ \frac{1}{p+2}\left(\sum_{k=1}^{d-i-1}\frac{1}{k(k+1)^2} + \sum_{k=1}^{i-1}  \frac{1}{k^2}- \sum_{k=d-i+1}^{d-1}  \frac{1}{k^2}\right)\right\}\frac{1}{(x_{i+1}-x_i)^{p+2}},
	\end{align*}
	where we shall use from now on the convention that $\sum_{i=m}^n a_i = 0$ if $m > n$. 
	By following a similar argument, we can bound $S_2$ as
	\begin{align*}
		S_2&\leq \sum_{i=1}^{d-1}\left\{ \frac{p}{p+2} \left(1-\frac{1}{i} \right) + \right.\\
		&\quad \left.+\frac{1}{p+2}\left(\sum_{k=1}^{i-1}\frac{1}{k(k+1)^2} + \sum_{k=1}^{d-i-1}  \frac{1}{k^2}- \sum_{k=i+1}^{d-1}  \frac{1}{k^2}\right)\right\}\frac{1}{(x_{i+1}-x_i)^{p+2}}.
	\end{align*}
	Therefore 
	\begin{align} \label{aux_1}
		S_1+ S_2 \leq \sum_{i=1}^{d-1} \frac{ \varphi_i}{(x_{i+1}-x_i)^{p+2}},
	\end{align}
	where $\varphi_i$ is defined by 
	\begin{align*}
		\varphi_i =& \frac{p}{p+2}\left(2 - \frac{1}{i} - \frac{1}{d-i}\right)\\
		&+ \frac{1}{p+2} \left\{ \sum_{k=1}^{i-1} \left( \frac{1}{k^2} + \frac{1}{k(k+1)^2}\right) +  \sum_{k=1}^{d-i-1} \left( \frac{1}{k^2} + \frac{1}{k(k+1)^2}\right) -  \sum_{k=i+1}^{d-1}  \frac{1}{k^2} -  \sum_{k=d-i+1}^{d-1}  \frac{1}{k^2} \right\}.
	\end{align*}
	From the fact that 
	$$\sum_{k=1}^{n}\Big(\frac{1}{k^2}+\frac{1}{k(k+1)^2}\Big)=\sum_{k=1}^{n} \Big( \frac{1}{k^2}+\frac{1}{k(k+1)} - \frac{1}{(k+1)^2}\Big) = 2-\frac{1}{n+1}-\frac{1}{(n+1)^2},$$
	we have 
	\begin{align*}
		\varphi_i & = \frac{p}{p+2}\left(2 - \frac{1}{i} - \frac{1}{d-i}\right)
		+ \frac{1}{p+2} \left\{ 4- \frac{1}{i} - \frac{1}{d-i} - \sum_{k=i}^{d-1} \frac{1}{k^2} - \sum_{k=d-i}^{d-1} \frac{1}{k^2}  \right\}\\
		& \leq \frac{p}{p+2} \left(2 - \frac{1}{i} - \frac{1}{d-i} \right) + \frac{1}{p+2} \left\{ 4- \frac{1}{i} - \frac{1}{d-i} - \sum_{k=i}^{d-1} \frac{1}{k(k+1)} - \sum_{k=d-i}^{d-1} \frac{1}{k(k+1)}  \right\} \\
		&=  \frac{p}{p+2} \left(2 - \frac{1}{i} - \frac{1}{d-i} \right) + \frac{1}{p+2} \left\{ 4- \frac{2}{i} - \frac{2}{d-i} + \frac2d  \right\}.
	\end{align*}
	By using the estimate  $\frac{1}{i} + \frac{1}{d-i} \geq \frac 4d$ we get 
	\begin{equation*}
		\varphi_i \leq 2 - \frac 3d, \quad i=1, \ldots, d-1.
	\end{equation*}
	This estimate together with \eqref{aux_1} implies the desired result.
\end{proof}

Recall that $X_{i,j}(t)=X_i(t)-X_j(t)$. Then, for $i>j$, we have
\begin{align*}
	X_{i,j}(t)
	=&X_{i,j}(0)
	+\int_{0}^{t} \frac{2\gamma}{X_{i,j}(s)} \mathrm{d}s
	-\int_{0}^{t} \sum_{k \neq i,j} \frac{\gamma X_{i,j}(s)}{X_{i,k}(s) X_{j,k}(s)} \mathrm{d}s\\
	&+\int_{0}^{t} \left\{ b_i(X_i(s))-b_j(X_j(s)) \right\} \mathrm{d}s
	+\sum_{k=1}^{d}\int_{0}^{t} \left\{ \sigma_{i,k}(X(s))-\sigma_{j,k}(X(s))\right\} \mathrm{d} W_k(s).
\end{align*}

For each $N>0$, we define the stopping time 
\begin{equation} \label{def:tauN}
	\tau_N:=\inf\{s >0 : \inf_{1\leq i \leq d-1} X_{i+1,i}(s) \leq 1/N \text{ or } \sup_{i=1,\ldots,d}|X_i(s)| \geq N\}.
\end{equation}
It is clear that $\tau_N \uparrow \tau$ as $N \to \infty$.

Before stating the next lemma, we recall that $\sigma_d^2:=\displaystyle \sup_{i=1,\ldots,d} \sup_{x \in \real^d} \sum_{k=1}^d \sigma_{i,k}(x)^2$.
\begin{Lem}\label{IM_0}
	Suppose that Assumption \ref{Ass_2} holds.
	Assume that $\frac{3\gamma}{d\sigma_d^2} \geq 1$, $p\in [0,\frac{3\gamma}{d\sigma_d^2}-1]$, $T>0$ and $\e[X_{i+1,i}(0)^{-p}]<\infty$ for each $i=1,\ldots,d-1$.
	Then it holds that 
	\begin{align*}
		\sum_{i=1}^{d-1} \sup_{0\leq t \leq T}\e\left[ X_{i+1,i}(t \wedge \tau)^{-p} \right]
		\leq \Big(\sum_{i=1}^{d-1} \e[X_{i+1,i}(0)^{-p}]\Big) e^{pT\|b\|_{Lip}}.
	\end{align*}
	In particular, 
	$$\sup_{i \not = j}  \sup_{0\leq t \leq T} \e[X_{i,j}(t\wedge \tau)^{-p}] 	\leq \Big(\sum_{i=1}^{d-1} \e[X_{i+1,i}(0)^{-p}]\Big) e^{pT\|b\|_{Lip}}.$$
\end{Lem}
\begin{proof}
	By using It\^o's formula, we have
	\begin{align*}
		&	X_{i+1,i}(t \wedge \tau_N)^{-p}\\
		=&X_{i+1,i}(0)^{-p}
		+\int_{0}^{t \wedge \tau_N}\left\{
		\frac{-2p\gamma}{X_{i+1,i}(s)^{p+2}}
		+ \frac{p\gamma}{X_{i+1,i}(s)^{p}} \sum_{k \neq i,i+1} \frac{1}{X_{i,k}(s) X_{i+1,k}(s)}
		\right\} \mathrm{d}s\\
		&-\int_{0}^{t \wedge \tau_N}
		\frac{p\{b_{i+1}(X_{i+1}(s))-b_i(X_i(s))\}}{X_{i+1,i}(s)^{p+1}}
		\mathrm{d}s\\
		&+\sum_{k=1}^{d}\int_{0}^{t \wedge \tau_N}
		\frac{p(p+1) |\sigma_{i+1,k}(X(s))-\sigma_{i,k}(X(s))|^2}{2X_{i+1,i}(s)^{p+2}} \mathrm{d}s\\
		&-\sum_{k=1}^{d}\int_{0}^{t \wedge \tau_N} \frac{p\left\{\sigma_{i+1,k}(X(s))-\sigma_{i,k}(X(s))\right\} }{X_{i+1,i}(s)^{p+1}}\mathrm{d}W_{k}(s).
	\end{align*}
	Since for each $i=1,\ldots,d$,
	\begin{align*}
		\int_{0}^{t}
		\left| \frac{ \left\{\sigma_{i+1,k}(X(s))-\sigma_{i,k}(X(s))\right\}}{X_{i+1,i}(s)^{p+1}} \1_{\{s \leq \tau_N\}} \right|^2\mathrm{d}s
		\leq 2\sigma_d^2 N^{2(p+1)}t,
	\end{align*}
	thus the expectations of the above stochastic integrals are zero.
	Moreover, since $(p+1)d \sigma_d^2 < 3\gamma $, by applying Lemma \ref{Lem:aux_1} we obtain 
	\begin{align*}
		\sum_{i=1}^{d-1} \e[X_{i+1,i}(t \wedge \tau_N)^{-p}]
		\leq& \sum_{i=1}^{d-1} \e[X_{i+1,i}(0)^{-p}] 
		-  \sum_{i=1}^{d-1}\e\left[\int_{0}^{t\wedge \tau_N}
		\frac{p\{b_{i+1}(X_{i+1}(s))-b_{i}(X_i(s))\}}{X_{i+1,i}(s)^{p+1}}
		\mathrm{d}s\right].
	\end{align*}
	Since $b_{i+1} \geq b_i$ and $b_i$ is Lipschitz continuous, we have
	\begin{align*}
		\sum_{i=1}^{d-1} \e[X_{i+1,i}(t \wedge \tau_N)^{-p}]
		\leq&  \sum_{i=1}^{d-1} \e[X_{i+1,i}(0)^{-p}]
		- \sum_{i=1}^{d-1}\e\left[\int_{0}^{t\wedge \tau_N}
		\frac{p\{b_{i}(X_{i+1}(s))-b_i(X_i(s))\}}{X_{i+1,i}(s)^{p+1}}\mathrm{d}s\right] \notag\\
		\leq& \sum_{i=1}^{d-1}\e[X_{i+1,i}(0)^{-p}]
		+p \|b\|_{Lip}\int_{0}^{t}
		\sum_{i=1}^{d-1}\e[X_{i,i+1}(s\wedge \tau_N)^{-p}]
		\mathrm{d}s.
	\end{align*}
	Using Gronwall's inequality, we get 
	$$\sum_{i=1}^{d-1} \e[X_{i+1,i}(t \wedge \tau_N)^{-p}] \leq \Big( \sum_{i=1}^{d-1}\e[X_{i+1,i}(0)^{-p}]
	\Big) e^{pt \|b\|_{Lip}}.$$
	Let $N\to \infty$ we conclude the proof of the Lemma.
\end{proof}

\begin{Lem}\label{momen_Hol_0}
	Suppose that Assumption \ref{Ass_2} holds.
	Assume that	$\frac{3\gamma}{d\sigma_d^2} \geq 2$, $p\in [1,\frac{3\gamma}{d\sigma_d^2}-1]$, $T>0$, $\e[|X(0)|^{p}]<\infty$ and $\e[X_{i+1,i}(0)^{-p}]<\infty$ for each $i=1,\ldots,d-1$.
	Then there exists a finite constant $C$ such that
	\begin{align}\label{momen_Hol_1}
		\sup_{0\leq t \leq T}\sup_{1\leq i \leq d} \e[|X_{i}(t\wedge \tau)|^{p}] \leq C.
	\end{align}
\end{Lem}
\begin{proof}
	Since $|b_i(x)|\leq |b_i(0)|+\|b\|_{Lip}|x|$ for any $x \in \real$, we have
	\begin{align*}
		|X_i(t \wedge \tau_N)| \leq& |X_i(0)|  + |b_i(0)|t + \|b\|_{Lip}\int_0^{t\wedge \tau_N} |X_i(s)|\rd s \\
		&+ \sum_{j\not = i}\int_0^{t\wedge \tau_N} \frac{\gamma}{|X_{i,j}(s)|}\rd s+  \sum_{j=1}^d \Big| \int_0^{t\wedge \tau_N} \sigma_{i,j}(X(s))\rd W_j(s)\Big|.
	\end{align*}
	A simple calculation yields
	\begin{align*}
		\frac{|X_i(t \wedge \tau_N)|^p}{(2d+2)^{p-1}} 
		\leq &
		|X_i (0)|^p +  |b_i(0)|^p t^p  + \|b\|_{Lip}^p t^{p-1} \int_0^{t} |X_i(s\wedge \tau_N)|^{p}\rd s \\
		& + t^{p-1} \sum_{j\not = i}  \int_0^{t} \frac{ \gamma^p}{|X_{i,j}(s\wedge \tau_N)|^p}\rd s
		+ \sum_{j=1}^{d} \Big| \int_0^{t\wedge \tau_N} \sigma_{i,j}(X(s))\rd W_j(s)\Big|^{p}.
	\end{align*}
	Denote $C_0 = \Big(\sum_{i=1}^{d-1} \e[X_{i+1,i}(0)^{-p}]\Big) e^{pT\|b\|_{Lip}}$.
	From Lemma \ref{IM_0}, Burkholder-Davis-Gundy's inequality and the boundedness of $\sigma_{i,j}$, by taking expectation,
	\begin{align*}
		\frac{\e[|X_i(t \wedge \tau_N)|^p]}{(2d+2)^{p-1}} 
		&\leq \e[|X_i(0)|^p] + |b_i(0)|^p t^p  \\
		&\quad + \|b\|_{Lip}^p t^{p-1} \int_0^{t} \e [|X_i(s\wedge \tau_N)|^{p}]\rd s  + (d-1)t^{p} \gamma^p C_0  + c(p)dt^{p/2}\sigma_d^p.
	\end{align*}
	It then follows from  Gronwall's inequality that  $\e[|X_i(t \wedge \tau_N)|^p]$ is bounded by 
	$$ (2d+2)^{p-1}\Big(  \e[|X_i(0)|^p] + |b_i(0)|^p t^p + (d-1)t^{p} \gamma^p C_0  + c(p)dt^{p/2}\sigma_d^p\Big)e^{ (2d+2)^{p-1} \|b\|_{Lip}^p t^{p-1}}.$$
	Let $N\to \infty$, we obtain 
	\begin{align*}
		\e[|X_i(t \wedge \tau)|^p] 
		\leq & (2d+2)^{p-1}\Big(  \e[|X_i(0)|^p] + |b_i(0)|^p t^p \\
		&+ (d-1)t^{p} \gamma^p C_0  + c(p)dt^{p/2}\sigma_d^p\Big)e^{ (2d+2)^{p-1} \|b\|_{Lip}^p t^{p-1}}.
	\end{align*}
	This implies  the  assertion of Lemma \ref{momen_Hol_0}.
\end{proof}
The main result of this section reads as follows.
\begin{Thm} \label{momen_Hol_0x}
	Suppose that Assumption \ref{Ass_2} holds.
	Assume that $\frac{3\gamma}{d\sigma_d^2} \geq 2$, $p\in [1,\frac{3\gamma}{d\sigma_d^2}-1]$, $\e[|X(0)|^{p}]<\infty$ and $\e[X_{i+1,i}(0)^{-p}]<\infty$ for each $i=1,\ldots,d-1$.
	Then the equation \eqref{def:X} has a unique strong solution $X(t)$ such that $X(t) \in \Delta_d$ almost surely for all $t>0$.
	Moreover, for any $T>0$, there exists a finite constant $C$ such that for any $0\leq s<t\leq T$
	\begin{align}\label{momen_Hol_2}
		\sup_{i=1,\ldots,d}\e[|X_i(t)-X_i(s)|^{p}] \leq C(t-s)^{p/2}.
	\end{align}
	
\end{Thm}

\begin{proof}
	By applying  Lemma \ref{IM_0} and Lemma \ref{momen_Hol_0} with $p=1$ we deduce that $\tau= \infty$, which implies that the equation \eqref{def:X} has a unique global strong solution $X(t)$ whose value is in $\Delta_d$ for all $t>0$.
	
	Now we consider the second statement \eqref{momen_Hol_2}.
	For any $0\leq s < t \leq T$,
	\begin{align*}
		\frac{|X_i(t) - X_i(s)|^p}{(2d+1)^{p-1}} & \leq    
		| b_i(0)|^p (t-s)^{p} + \|b\|_{Lip}^p (t-s)^{p-1} \int_s^{t} |X_i(u)|^p \rd u \\
		&+ \sum_{j\not = i}(t-s)^{p-1} \int_s^{t} \frac{\gamma^p}{|X_{i,j}(u)|^p}\rd u+  \sum_{j=1}^d \Big| \int_s^{t} \sigma_{i,j}(X(u))\rd W_j(u)\Big|^p.
	\end{align*}
	It follows from   Lemma \ref{IM_0}, estimate \eqref{momen_Hol_1} and Burkholder-Davis-Gundy's inequality that 
	\begin{align*}
		\e[|X_i(t)-X_i(s)|^p]
		&\leq C(t-s)^{p}
		+\sum_{j=1}^d c(p) \e\Big[\Big( \int_s^{t} \sigma_{i,j}^2(X(u)) \rd u \Big)^{p/2}\Big] \\
		&\leq C (t-s)^{p/2}.
	\end{align*}
	This concludes the second assertion \eqref{momen_Hol_2}.
\end{proof}

\begin{Rem}
	Under the assumption of Theorem \ref{momen_Hol_0x} it is straightforward to verify that  the Hypothesis \ref{Ip} holds with $\hat{p}=p$.
\end{Rem}

\begin{Rem}
	The existence and uniqueness of  non-colliding solution in this paper are established under stricter conditions on $\gamma/\sigma^2$ than  in \cite{RogerShi, CepaLepingle} and \cite{GrMa14}. Note that these papers only considered  a particular case of  systems \eqref{def:X} where each coordinate $X_i$ is driven by a  independent Brownian motion and $\sigma_{i,j}(X(t)) = \delta_{i,j}\sigma_i(X_i(t))$, where $\delta_{i,j}$ is the Dirac delta function. Thanks to that stricter condition, the existence and uniqueness can be proven for a more general class of equations where the driving Brownian motions of each exponent can be correlated. More importantly, that condition allows us to obtain the moment estimation \eqref{momen_Hol_2} which is the key to study the strong rate of convergence for the discrete approximation for equation  \eqref{def:X}.
\end{Rem}


\subsection{Brownian particles with nearest neighbor repulsion}

In this section we consider the process  $X=(X_1,\ldots,X_d)$  given by the following SDEs
\begin{equation} \label{repX}
	\begin{cases} \rd X_1(t) = \left\{ \frac{\gamma}{X_1(t)-X_2(t)} + b_1(X_1(t))\right\} \rd t + \sum_{j=1}^d  \sigma_{1,j}(X(t))\rd W_j(t),\\
		\rd X_i(t) = \left\{ \frac{\gamma}{X_i(t)-X_{i-1}(t)} +\frac{\gamma}{X_i(t)-X_{i+1}(t)} + b_i(X_i(t))\right\}\rd t + \sum_{j=1}^d  \sigma_{i,j}(X(t))\rd W_j(t),\\
		\hspace{9cm} i=2,\ldots, d-1,\\
		\rd X_d(t) = \left\{ \frac{\gamma}{X_d(t)-X_{d-1}(t)} + b_d(X_{d}(t))\right\} \rd t + \sum_{j=1}^d  \sigma_{d,j}(X(t))\rd W_j(t),
	\end{cases}
\end{equation} 
with $X(0) \in \Delta_d$. Let the Assumptions (A1)--(A4) hold. Since the coefficients of equation \eqref{repX} are locally Lipschitz continuous in $\Delta_d$, given $X(0) \in \Delta_d$, equation \eqref{repX} has a unique strong local solution up to the stopping time $\tau$ defined by \eqref{def:tau}.

\begin{Rem}
	These kind of systems were studied in \cite{GrMa14, RoVa, Lep}. In particular, \cite{GrMa14} considered the following SDEs
	\begin{equation*} 
		\begin{cases} \rd X_1(t) =  \frac{\gamma}{X_1(t)-X_2(t)} \rd t +   \sigma_{1}(X_1(t))\rd W_1(t),\\
			\rd X_i(t) = \left\{ \frac{\gamma}{X_i(t)-X_{i-1}(t)} +\frac{\gamma}{X_i(t)-X_{i+1}(t)} \right\} \rd t +  \sigma_{i}(X_i(t))\rd W_i(t),  i=2,\ldots, d-1,\\
			\rd X_d(t) =  \frac{\gamma}{X_d(t)-X_{d-1}(t)} \rd t +  \sigma_{d}(X_d(t))\rd W_d(t).
		\end{cases}
	\end{equation*}
	It is shown that the system has a unique strong solution with no collisions and no explosions if $d=3$, $\gamma \geq \frac 34$ and $|\sigma_i|\leq 1$. 
\end{Rem}

In the following, we apply the method introduced in the previous sections, which is essentially different from the one in \cite{GrMa14}, to study the existence, uniquess, non-collision and non-explosions of the solution to the general equation \eqref{repX}.

\begin{Lem} \label{Lem:aux_2}
	For any $d\geq 3$ and $p \geq 0$, there exists a constant $\ochi(d,p) < 2 $ such that 
	\begin{align*}
		\sum_{i=1}^{d-2}&
		\left\{
		\frac{1}{(x_{i+2}-x_{i+1})(x_{i+1}-x_i)^{p+1}} + \frac{1}{(x_{i+2}-x_{i+1})^{p+1} (x_{i+1}-x_i)}
		\right\}\\
		&	\leq \ochi(d,p)\sum_{i=1}^{d-1} \frac{1}{(x_{i+1}-x_i)^{p+2}},
	\end{align*}
	for any $(x_1,\ldots, x_d) \in \Delta_d$.
\end{Lem}

\begin{proof}
	Denote 
	$L= \sum_{i=1}^{d-1} \frac{1}{(x_{i+1}-x_i)^{p+2}}$ 
	and 
	$\xi_i = \frac{1}{(x_{i+1}-x_i) L^{1/(p+2)}}$. 
	We have $\sum_{i=1}^{d-1} \xi_i^{p+2} = 1$.
	Denote $S^+_{p,d-1} := \{\xi=(\xi_1, \ldots, \xi_{d-1})\in \real_+^{d-1}\ : \   \sum_{i=1}^{d-1} \xi_i^{p+2} = 1 \}$,
	and 
	$\ochi(d,p) = \sup_{(\xi_1, \ldots, \xi_{d-1}) \in S^+_{p,d-1}} \sum_{i=1}^{d-2} (\xi_{i+1}\xi_i^p +\xi_{i+1}^p\xi_i).$
	Since  $S^+_{p,d-1}$
	is a compact subset of $\real^{d-1}$, the supremum is attainable. On the other hand, for any $\xi \in S^+_{p,d-1}$,
	\begin{align*} 
		2 - \sum_{i=1}^{d-2} (\xi_{i+1}\xi_i^{p+1} +\xi_{i+1}^{p+1}\xi_i) &= 2 \sum_{i=1}^{d-1} \xi_i^{p+2}  - \sum_{i=1}^{d-2} (\xi_{i+1}\xi_i^{p+1} +\xi_{i+1}^{p+1}\xi_i) \\
		&= \sum_{i=1}^{d-2} (\xi_i - \xi_{i+1})(\xi_i^{p+1} - \xi_{i+1}^{p+1}) + \xi_1^{p+2} + \xi_{d-1}^{p+2}.
	\end{align*}
	The last term is strictly positive since  for any  non-negative constants $a$ and $b$ the quantity $(a-b)(a^{p+1}-b^{p+1})$ is non-negative and it equals to zeros if and only if $a = b$. 
	This implies the desired result.
\end{proof}

We denote $X_{i+1,i}(t)=X_{i+1}(t)-X_i(t)$. Let $\tau_N$ be defined as in \eqref{def:tauN}.

\begin{Lem}\label{zIM_0}
	Suppose that Assumption \ref{Ass_2} holds.
	Let $p$ be a positive number satisfying $\frac{\gamma}{2\sigma_d^2} \geq \frac{p+1}{2-\ochi(d,p)}$.
	Suppose that $\e[X_{i+1,i}(0)^{-p}]<\infty$ for each $i=1,\ldots,d-1$.
	Then for any $T>0$, it holds that 
	\begin{align*}
		\sum_{i=1}^{d-1} \sup_{0\leq t \leq T}\e\left[ X_{i+1,i}(t \wedge \tau)^{-p} \right]
		\leq \Big(\sum_{i=1}^{d-1} \e[X_{i+1,i}(0)^{-p}]\Big) e^{pT\|b\|_{Lip}}.
	\end{align*}
	In particular, 
	$$\sup_{i < j}  \sup_{0\leq t \leq T} \e[X_{j,i}(t\wedge \tau)^{-p}] 	\leq \Big(\sum_{i=1}^{d-1} \e[X_{i+1,i}(0)^{-p}]\Big) e^{pT\|b\|_{Lip}}.$$
\end{Lem}
\begin{proof}
	By using It\^o's formula, for each $2\leq i \leq d-2$, we have 
	\begin{align*}
		&	X_{i+1,i}(t \wedge \tau_N)^{-p}\\
		=&X_{i+1,i}(0)^{-p}
		+\int_{0}^{t \wedge \tau_N} \frac{p\gamma}{X_{i+1,i}(s)^{p+1}} \left\{
		\frac{-2}{X_{i+1,i}(s)}
		+  \frac{1}{X_{i+2,i+1}(s)} + \frac{1}{X_{i,i-1}(s)} \right\} \mathrm{d}s\\
		&-\int_{0}^{t \wedge \tau_N}
		\frac{p\{b_{i+1}(X_{i+1}(s))-b_i(X_i(s))\}}{X_{i+1,i}(s)^{p+1}}
		\mathrm{d}s \\
		&+\sum_{k=1}^{d}\int_{0}^{t \wedge \tau_N}
		\frac{p(p+1) |\sigma_{i+1,k}(X(s))-\sigma_{i,k}(X(s))|^2}{2X_{i+1,i}(s)^{p+2}} \mathrm{d}s\\
		&-\sum_{k=1}^{d}\int_{0}^{t \wedge \tau_N} \frac{p\left\{\sigma_{i+1,k}(X(s))-\sigma_{i,k}(X(s))\right\} }{X_{i+1,i}(s)^{p+1}}\mathrm{d}W_{k}(s).
	\end{align*}
	In addition, we have 
	\begin{align*}
		&	X_{2,1}(t \wedge \tau_N)^{-p}\\
		=&X_{2,1}(0)^{-p}
		+\int_{0}^{t \wedge \tau_N} \frac{p\gamma}{X_{2,1}(s)^{p+1}} \left\{ \frac{-2}{X_{2,1}(s)}
		+  \frac{1}{X_{3,2}(s)}  \right\} \mathrm{d}s\\
		&-\int_{0}^{t \wedge \tau_N}
		\frac{p\{b_{2}(X_{2}(s))-b_1(X_1(s))\}}{X_{2,1}(s)^{p+1}}
		\mathrm{d}s\\
		&+\sum_{k=1}^{d}\int_{0}^{t \wedge \tau_N}
		\frac{p(p+1) |\sigma_{2,k}(X(s))-\sigma_{1,k}(X(s))|^2}{2X_{2,1}(s)^{p+2}} \mathrm{d}s\\
		&-\sum_{k=1}^{d}\int_{0}^{t \wedge \tau_N} \frac{p\left\{\sigma_{2,k}(X(s))-\sigma_{1,k}(X(s))\right\} }{X_{2,1}(s)^{p+1}}\mathrm{d}W_{k}(s),
	\end{align*}
	and
	\begin{align*}
		&	X_{d,d-1}(t \wedge \tau_N)^{-p}\\
		=&X_{d,d-1}(0)^{-p}
		+\int_{0}^{t \wedge \tau_N} \frac{p\gamma}{X_{d,d-1}(s)^{p+1}} \left\{
		\frac{-2}{X_{d,d-1}(s)}
		+ \frac{1}{X_{d-1,d-2}(s)} \right\} \mathrm{d}s\\
		&-\int_{0}^{t \wedge \tau_N}
		\frac{p\{b_{d}(X_{d}(s))-b_{d-1}(X_{d-1}(s))\}}{X_{d,d-1}(s)^{p+1}}
		\mathrm{d}s\\
		&+\sum_{k=1}^{d}\int_{0}^{t \wedge \tau_N}
		\frac{p(p+1) |\sigma_{d,k}(X(s))-\sigma_{d-1,k}(X(s))|^2}{2X_{d,d-1}(s)^{p+2}} \mathrm{d}s\\
		&-\sum_{k=1}^{d}\int_{0}^{t \wedge \tau_N} \frac{p\left\{\sigma_{d,k}(X(s))-\sigma_{d-1,k}(X(s))\right\} }{X_{d,d-1}(s)^{p+1}}\mathrm{d}W_{k}(s).
	\end{align*}
	Since $\frac{\gamma}{d\sigma_d^2} \geq \frac{p+1}{2-\ochi(d,p)}$, by applying Lemma \ref{Lem:aux_2} we obtain
	\begin{align*}
		&\sum_{i=1}^{d-1} \e[X_{i+1,i}(t \wedge \tau_N)^{-p}]\\
		&\leq \sum_{i=1}^{d-1} \e[X_{i+1,i}(0)^{-p}] 
		-  \sum_{i=1}^{d-1}\e\left[\int_{0}^{t\wedge \tau_N}
		\frac{p\{b_{i+1}(X_{i+1}(s))-b_{i}(X_i(s))\}}{X_{i+1,i}(s)^{p+1}}
		\mathrm{d}s\right].
	\end{align*}
	The proof is concluded by following the same argument as in the proof of Lemma \ref{IM_0}. 
\end{proof}

By using Lemma \ref{zIM_0} and adapting the argument of the previous sections, we can show the following result.
\begin{Thm} \label{zmomen_Hol_0x}
	Suppose that Assumption \ref{Ass_2} holds.
	Assume that there exists constant $p\geq 1$ such that $\frac{\gamma}{2\sigma_d^2} \geq \frac{p+1}{2-\ochi(d,p)}$, $\e[|X(0)|^p] < \infty$ and $\e[X_{i+1,i}(0)^{-p}]<\infty$ for each $i=1,\ldots,d-1$.
	Then the equation \eqref{repX} has a unique strong solution $X(t)$ such that $X(t) \in \Delta_d$ almost surely for all $t>0$.
	Moreover, Hypothesis \ref{Ip} holds for $\hat{p} = p$.
\end{Thm}

\section{Numerical approximation for system of equation \eqref{eqn_xi}} \label{Sec:4}

In this section, we discuss how to approximate the solution of the system of equations \eqref{eqn_xi}. 
Denote $x_i = \xi_{i+1} - \xi_i$ and $a_i = b_{i+1} - b_i$. We can rewrite \eqref{eqn_xi} as the following system of equations of variables $x_i, 1\leq i \leq d-1$,
\begin{align}\label{L0}
	x_i = a_i + &\frac{2c_{i,i+1}}{x_i} + \sum_{j< i} \Big( \frac{c_{i+1,j}}{x_j + \ldots + x_i} -  \frac{c_{i,j}}{x_j + \ldots + x_{i-1}}\Big) \\
	&- \sum_{j> i+1} \Big( \frac{c_{i+1,j}} {x_{i+1}+\ldots + x_{j-1}} - \frac{c_{i,j}}{x_{i} + \ldots + x_{j-1}}\Big). \notag
\end{align}
It is clear that there is an one-to-one correspondence between $(\xi_i)_{1 \leq i \leq d}$ and $(x_i)_{1 \leq i \leq d-1}$. Note that since both systems \eqref{eqn_xi} and  \eqref{L0} are highly non-linear and very stiff, it is very hard to find an effective numerical approximation scheme for them in the general case. In the following we will construct some iterative schemes for the system of equations \eqref{L0} and show their convergence in some particular cases. 

We first consider the case that $c_{i,j} = 0$ for all $i,j$ satisfying $|i-j| \geq 2$, which  corresponds to the system of Brownian particles with nearest neighbor repulsion. We denote $c_{i,i+1} = c_i$ and $k=d-1$ for the sake of simplicity.

\begin{Prop} \label{iter1}
	Let $a  = (a_1, \ldots, a_k) \in \mathbb{R}^k$ and  $c_i>0$ for all $i=1,\ldots, k$. The following system of equations 
	\begin{equation} \label{L1}
		\begin{cases}
			x_1 - \frac{2c_{1}}{x_1} = a_1 - \frac{c_{2}}{x_2} \\
			x_i - \frac{2c_{i}}{x_i} = a_i - \frac{c_{i-1}}{ x_{i-1}} - \frac{c_{i+1}}{x_{i +1}} , \quad i=2, \ldots, k-1,\\
			x_k - \frac{2c_{k}}{x_k} = a_k - \frac{c_{k-1}}{x_{k-1}}
		\end{cases}
	\end{equation}
	has a unique solution $(x_1^*, \ldots, x_k^*) \in \mathbb{R}_+^k$. Moreover, if we consider the sequence
	\begin{align*}
		& x^{(0)}_i = \frac{1}{2}( a_i + \sqrt{a_i^2 + 8c_i}), \quad i = 1, \ldots, k\\
		& \begin{cases} 
			x^{(n+1)}_1 = \dfrac{1}{2} \Big( a_1 - \frac{c_2}{x^{(n)}_2} + \sqrt{\Big(a_1 - \frac{c_2}{x^{(n)}_2 }\Big)^2+8c_1}\Big)\\
			x^{(n+1)}_i = \dfrac{1}{2} \Big( a_i - \frac{c_{i-1}}{x^{(n)}_{i-1}}-\frac{c_{i+1}}{x^{(n)}_{i+1}}  + \sqrt{\Big(a_i - \frac{c_{i-1}}{x^{(n)}_{i-1}}-\frac{c_{i+1}}{x^{(n)}_{i+1}} \Big)^2 + 8c_i}\Big),  i=2,\ldots, k-1, n \geq 0.\\ 
			x^{(n+1)}_k=  \dfrac{1}{2} \Big( a_k - \frac{c_{k-1}}{x^{(n)}_{k-1}} + \sqrt{\Big(a_k - \frac{c_{k-1}}{x^{(n)}_{k-1}} \Big)^2+8c_k}\Big)
		\end{cases}
	\end{align*} 
	Then for each $i=1, \ldots, k$, the sequence $x^{(n)}_i$ decreases to $x^*_i$ as $n$ tends to infinity.
\end{Prop}
\begin{proof}
	The existence and uniqueness of solution of \eqref{L1} is a direct consequence of Proposition \ref{IEM_sol}. It is clear that 
	$x_i^{(0)} - \frac{2c_i}{x_i^{(0)}} = a_i,$
	and 
	$$\begin{cases}
	x^{(n+1)}_1 - \frac{2c_1}{x^{(n+1)}_1} = a_1 - \frac{c_2}{x^{(n)}_2} \\
	x^{(n+1)}_i - \frac{2c_i}{x^{(n+1)}_i} = a_i - \frac{c_{i-1}}{ x^{(n)}_{i-1}} - \frac{c_{i+1}}{x^{(n)}_{i +1}} , \quad i=2, \ldots, k-1\\
	x^{(n+1)}_k - \frac{2c_k}{x^{(n+1)}_k} = a_k - \frac{c_{k-1}}{x^{(n)}_{k-1}}.
	\end{cases}$$
	Note that if $c>0$ then the mapping $x\mapsto x - \frac{2c}{x}$ is strictly increasing on $(0,+\infty)$. 
	Since $x^{(n+1)}_i - \frac{2c_i}{x^{(n+1)}_i} < a_i = x_i^{(0)} - \frac{2c_i}{x_i^{(0)}}$, we have   
	$x^{(n+1)}_i < x^{(0)}_i$ for all $n \geq 0, 1\leq i \leq k$. In particular, we have $x^{(1)}_i < x^{(0)}_i$. Using the induction method, we obtain that for each $i$, the sequence $(x^{(n)}_i)_{n\geq 0}$ is a decreasing sequence of positive numbers. Indeed, suppose that $x^{(n+1)}_i < x^{(n)}_i$ for all $i = 1, 2, \ldots, k$. Then for any $i=2, \ldots, k-1$, it holds
	\begin{align*}
		&\left(x_i^{(n+1)} - \frac{2c_i}{x_i^{(n+1)} }\right)  - \left( x_i^{(n+2)} - \frac{2c_i}{x_i^{(n+2)} }\right)  \\
		&= c_{i-1} \left( \frac{1}{x^{(n+1)}_{i-1}} - \frac{1}{x^{(n)}_{i-1}}\right) + c_{i+1} \left(\frac{1}{x^{(n+1)}_{i+1}} - \frac{1}{x^{(n)}_{i+1}}\right) > 0,
	\end{align*}
	which implies $x_i^{(n+1)} > x_i^{(n+2)}$ for $i=2,\cdots, k-1.$ A similar argument yields that $x_i^{(n+1)} > x_i^{(n+2)}$ for $i=1$ and $i=k$ as well. 
	Therefore, for each $i$, sequence $(x^{(n)}_i)_{n\geq 0}$ converges to the desired limits $x^*_i$.
\end{proof}

Next we consider the system \eqref{eqn_xi} when $d= 3$.
\begin{Prop} \label{iter2}
	Let $a, b \in \mathbb{R}$. The following system of equations 
	$$\begin{cases}
	x - \frac 2x = a - \frac 1y + \frac{1}{x+y}\\ 
	y - \frac 2y = b - \frac 1x + \frac{1}{x+y}\\
	\end{cases}$$
	has a unique solution $(x^*,y^*) \in \real_+^2$. Moreover, if we consider the sequence 
	$$x_1 = \frac 12 (a + \sqrt{a^2+8}), \quad  y_1 = \frac{1}{2}\Big( b - \frac{|a|+\sqrt{2}}{2} + \sqrt{\big(b - \frac{|a|+\sqrt{2}}{2}\big)^2 + 6}\Big),$$
	and 
	$$\begin{cases} 
	x_{n+1} = \frac 12 \Big( a - \frac{1}{y_n} + \frac{1}{x_n + y_n} + \sqrt{ \big(a - \frac{1}{y_n} + \frac{1}{x_n + y_n}\big)^2 + 8}\Big),\\
	y_{n+1} = \frac 12 \Big( b - \frac{1}{x_n} + \frac{1}{x_n + y_n} + \sqrt{ \big(b - \frac{1}{x_n} + \frac{1}{x_n + y_n}\big)^2 + 8}\Big)
	\end{cases}$$
	Then
	$$\lim {x_n} = x^* \text{ and } \quad \lim y_n = y^*.$$
\end{Prop}
\begin{proof}
	Step 1: It is clear that 
	$$\begin{cases} 
	x_{n+1} - \frac {2}{x_{n+1}} = a - \frac {1}{y_n} + \frac{1}{x_n+y_n}\\ 
	y_{n+1} - \frac {2}{y_{n+1}} = b - \frac {1}{x_n} + \frac{1}{x_n+y_n}.
	\end{cases}$$
	Since $x_{n+1} - \frac {2}{x_{n+1}} < a $ then $x_{n+1} < \frac 12 \big( \sqrt{a^2+8} + a\big)$ for all $n \geq 1$. Similarly, $y_{n+1} < \frac 12 \big( \sqrt{b^2+8} + b\big)$ for all $n \geq 1$. 
	
	Step 2: We show that $y_1 < y_3$. Indeed,  we have 
	\begin{align*}
		y_3 - \frac{2}{y_3 } &> b - \frac{1}{x_2} \\
		&= b - \frac 14 \Big( \sqrt{\big(a - \frac{1}{y_1} + \frac{1}{x_1+y_1}\big)^2 + 8} - a + \frac{1}{y_1} - \frac{1}{x_1+y_1}\Big).
	\end{align*}
	Applying the simple estimate, $\sqrt{a^2 +b^2} \leq |a| + |b|$, we get 
	\begin{align*}
		\sqrt{\big(a - \frac{1}{y_1} + \frac{1}{x_1+y_1}\big)^2 + 8}  &\leq |a| + \frac{1}{y_1}+ \frac{1}{x_1+y_1} + \sqrt{8} . 
	\end{align*}
	Therefore 
	\begin{align*}
		y_3 - \frac{2}{y_3} > b - \frac{1}{2y_1} - \frac{|a|+\sqrt{2}}{2}  = y_1 - \frac{2}{y_1}.
	\end{align*}
	This implies $y_3 > y_1$. 
	
	Step 3: 
	Since the function $x \mapsto \frac 12 (x + \sqrt{x^2+8})$ is strictly increasing on $\mathbb{R}$, for any $k\geq 0$, we have the following relation.
	\begin{align}
		x_{k+2} < x_{k+4} &\Leftrightarrow \frac{-1}{y_{k+1}} + \frac{1}{x_{k+1}+y_{k+1}} <  \frac{-1}{y_{k+3}} + \frac{1}{x_{k+3}+y_{k+3}} \notag\\
		&\Leftrightarrow y_{k+1} + \frac{y_{k+1}^2}{x_{k+1}} < y_{k+3} + \frac{y_{k+3}^2}{x_{k+3}}. \label{app1}
	\end{align}
	Similarly, for any $k\geq 0$, we have 
	\begin{align} \label{app2}
		y_{k+2} > y_{k+4} \Leftrightarrow 
		x_{k+3} + \frac{x_{k+3}^2}{y_{k+3}} < x_{k+1}+ \frac{x_{k+1}^2}{y_{k+1}}.
	\end{align}
	Since $x_1> x_3>0$ and $y_3 > y_1>0$, it follows from the relations \eqref{app1} and \eqref{app2} with $k=0$ that $x_2< x_4$ and $y_2> y_4$. Using the relations \eqref{app1} and \eqref{app2} again with $k=1$ yields $x_3> x_5$ and $y_3< y_5$. By repeating this argument, we get 
	\begin{equation} \label{thutu}
		\begin{cases} &y_1 < y_3 < y_5 < y_7 <  \ldots \\
			& y_2 > y_4 > y_6 > y_8 > \ldots\\
			& x_1 > x_3 > x_5 > x_7 > \ldots\\
			& x_2 < x_4 < x_ 6 < x_8 < \ldots 
		\end{cases}
	\end{equation}

	
	Step 4: It follows from Step 1 and \eqref{thutu} that the sequences $(y_{2k+1}), (y_{2k}), (x_{2k+1}), (x_{2k})$ converge to non-negative constants $\hat{y}_1, \hat{y}_2, \hat{x}_1, \hat{x}_2$, respectively. Moreover, $\hat{y}_1, \hat{y}_2, \hat{x}_1, \hat{x}_2$ satisfy
	$$\begin{cases} 
	\hat{y}_1 - \frac{2}{\hat{y}_1}= b - \frac{1}{\hat{x}_2}+ \frac{1}{\hat{x}_2 + \hat{y}_2}\\
	\hat{y}_2 - \frac{2}{\hat{y_2}}= b - \frac{1}{\hat{x}_1}+ \frac{1}{\hat{x}_1 + \hat{y}_1}\\
	\hat{x}_1 - \frac{2}{\hat{x}_1}= a - \frac{1}{\hat{y}_2}+ \frac{1}{\hat{x}_2 + \hat{y}_2}\\
	\hat{x}_2 - \frac{2}{\hat{x}_2}= a - \frac{1}{\hat{y}_1}+ \frac{1}{\hat{x}_1 + \hat{y}_1}.
	\end{cases}$$
	The first two equations imply
	$$(\hat{y}_2 - \hat{y}_1) \Big( 1 + \frac{2}{\hat{y}_1 \hat{y}_2} - \frac{1}{(\hat{x}_1+\hat{y}_1)(\hat{x}_2+\hat{y}_2)}\Big) = (\hat{x}_2 - \hat{x}_1) \Big(\frac{1}{\hat{x}_1\hat{x}_2} - \frac{1}{(\hat{x}_1+\hat{y}_1)(\hat{x}_2+\hat{y}_2)}\Big) ,$$
	while the last two equations imply
	$$(\hat{x}_2 - \hat{x}_1) \Big( 1 + \frac{2}{\hat{x}_1 \hat{x}_2} - \frac{1}{(\hat{x}_1+\hat{y}_1)(\hat{x}_2+\hat{y}_2)}\Big) = (\hat{y}_2 - \hat{y}_1) \Big(\frac{1}{\hat{y}_1\hat{y}_2} - \frac{1}{(\hat{x}_1+\hat{y}_1)(\hat{x}_2+\hat{y}_2)}\Big).$$
	We show that  $\hat{x}_1 = \hat{x}_2$ and  $\hat{y}_1 = \hat{y_2}$ by contradiction method. Indeed,  suppose that  $\hat{x}_1 \not = \hat{x}_2$ then $\hat{y}_1 \not = \hat{y_2}$ and  
	\begin{align*}
		&\Big( 1 + \frac{2}{\hat{y}_1 \hat{y}_2} - \frac{1}{(\hat{x}_1+\hat{y}_1)(\hat{x}_2+\hat{y}_2)}\Big)  \Big( 1 + \frac{2}{\hat{x}_1 \hat{x}_2} - \frac{1}{(\hat{x}_1+\hat{y}_1)(\hat{x}_2+\hat{y}_2)}\Big) \\
		&\quad = \Big(\frac{1}{\hat{x}_1\hat{x}_2} - \frac{1}{(\hat{x}_1+\hat{y}_1)(\hat{x}_2+\hat{y}_2)}\Big)  \Big(\frac{1}{\hat{y}_1\hat{y}_2} - \frac{1}{(\hat{x}_1+\hat{y}_1)(\hat{x}_2+\hat{y}_2)}\Big) .
	\end{align*}
	This is not true since the term on left hand side is alway strictly greater than the one on right hand side. It means that  $\hat{x}_1  = \hat{x}_2$ and  $\hat{y}_1 = \hat{y_2}$.
\end{proof}
\begin{Rem}
	The iterative method in Proposition \ref{iter2} could be generalized to the case that $d\geq 3$ in a straightforward way. However, our simulation shows that that scheme may not converge when $d \geq 4$.  
\end{Rem}

\section*{Acknowledgements}

This work was supported by JSPS KAKENHI Grant Number 16J00894 and 17H06833, and by the
Vietnamese National Foundation for Science and Technology Development
(NAFOSTED) Grant Number 101.03-2017.308.
Part of this work was  carried out while H.N. was visiting Ritsumeikan University and  Vietnam Institute for Advanced Study in Mathematics (VIASM).
He would like to thank these institutes for their support.

\end{document}